\documentclass[12pt,reqno]{amsart}
\usepackage{amssymb,graphicx}

\numberwithin{equation}{section}

\newtheorem{theorem}{Theorem}
\newtheorem{lemma}[theorem]{Lemma}
\newtheorem{proposition}[theorem]{Proposition}
\newtheorem{corollary}[theorem]{Corollary}

\theoremstyle{remark}
\newtheorem{remark}[theorem]{Remark}

\newtheorem{conjecture}[theorem]{Conjecture}

\begin{document}

\title[Holomorphic projective connections on compact threefolds]{Holomorphic
projective connections on compact complex threefolds}

\author[I. Biswas]{Indranil Biswas}

\address{School of Mathematics, Tata Institute of Fundamental
Research, Homi Bhabha Road, Mumbai 400005, India}

\email{indranil@math.tifr.res.in}

\author[S. Dumitrescu]{Sorin Dumitrescu}

\address{Universit\'e C\^ote d'Azur, CNRS, LJAD, France}

\email{dumitres@unice.fr}



\begin{abstract}
We prove that a holomorphic projective connection on a complex projective threefold is either flat, or it is
a translation invariant holomorphic projective connection on an abelian threefold. In the second case, a generic 
translation invariant holomorphic affine connection on the abelian variety is \textit{not} projectively
flat. We also 
prove that a simply connected compact complex threefold with trivial canonical line
bundle does not admit any holomorphic projective connection.
\end{abstract}

\maketitle

\section{Introduction}

An important consequence of the uniformization theorem for Riemann surfaces is that any Riemann surface 
admits a holomorphic projective structure which is isomorphic either to the one-dimensional model ${\mathbb 
C}{\mathbb P}^1$, or to a quotient of the complex affine line $\mathbb C$ by a discrete group of translations, or 
to a quotient of the complex hyperbolic space ${\rm \mathbb H}_{\mathbb C}^1$ by a torsion-free discrete subgroup 
of ${\rm SU}(1,1) \,\simeq\, {\rm SL}(2, \mathbb R)$ \cite{Gu,StG}. In higher dimensions
compact complex manifolds do not, in general, admit any holomorphic  projective  structure.

Kobayashi and Ochiai in \cite{KO,KO1} classified compact K\"ahler--Einstein manifolds admitting a holomorphic 
projective connection. Their result says that the only examples of compact K\"ahler--Einstein manifolds admitting a 
holomorphic projective connection are the standard ones; we recall that the $n$--dimensional standard examples are 
the following: the complex projective space ${\mathbb C}{\mathbb P}^n$, all \'etale quotients of complex
$n$--tori and all
compact quotients of the complex hyperbolic $n$-space ${\rm \mathbb H}_{\mathbb C}^n$ by a torsion-free discrete 
subgroup of ${\rm SU}(n,1)$. All of these three types of manifolds are endowed with a standard flat 
holomorphic projective connection, i.e., a holomorphic projective structure, which is locally modeled on ${\mathbb 
C}{\mathbb P}^n$ (see Section \ref{section 2}).

Moreover, Kobayashi and Ochiai gave a classification of compact complex surfaces admitting holomorphic projective 
connections \cite{KO,KO1}. Their classification shows that all those compact complex surfaces also admit flat 
holomorphic projective connections. The geometry of flat holomorphic projective structures on compact complex 
surfaces was studied by Klingler in \cite{Kl2}. Subsequently, it was proved in \cite{Du1} that all
holomorphic (normal) projective connections on compact complex surfaces are flat.

Here we study the local geometry of holomorphic projective connections on compact complex manifolds 
of dimension three and higher. For defining holomorphic projective connections we adopt the 
terminology of \cite{Gu,MM}; these connections are holomorphic normal projective connections in the terminology of 
\cite{Ka,KO, JR1} (see Section \ref{section 2}).

For compact K\"ahler--Einstein manifolds, using the classification in \cite{KO,KO1}, and
generalizing, to the non-flat case, some results of Mok--Yeung and Klingler on flat
projective connections, \cite{MY,Kl1}, we prove the following (see Section \ref{pt1}):

\begin{theorem}\label{classification}
Let M be a compact K\"ahler--Einstein manifold of complex dimension $n\,>\, 1$ endowed with a holomorphic projective
connection. Then the following hold:
\begin{enumerate} 
\item either $M$ is the complex projective space ${\mathbb C}{\mathbb P}^n$ endowed with its standard flat projective 
connection;

\item or $M$ is a quotient of the complex hyperbolic space ${\rm \mathbb H}_{\mathbb C}^n$, by a discrete 
subgroup in ${\rm SU}(n,1)$, endowed with its induced standard flat projective connection;

\item or $M$ is an \'etale quotient of a compact complex $n$--torus endowed with the holomorphic projective 
connection induced by a translation invariant holomorphic torsionfree affine connection on
the universal cover ${\mathbb C}^n$.
For $n\,\geq\, 3$, the general translation invariant holomorphic torsionfree affine connection on ${\mathbb C}^n$ is not projectively flat.
\end{enumerate}
\end{theorem}

In particular, a holomorphic projective connection $\phi$ on a compact K\"ahler--Einstein manifold of complex 
dimension $n$ is either flat, or it is locally isomorphic to the projective connection induced by a translation 
invariant holomorphic affine connection on ${\mathbb C}^n$. In both cases, $\phi$ is locally homogeneous; 
more precisely, the local  projective Killing Lie algebra of the projective connection $\phi$ (see Section \ref{section 2}
for definition) contains a copy of the abelian Lie algebra ${\mathbb C}^n$ which is transitive on $M$.
 
The third case in Theorem \ref{classification} covers all compact K\"ahler manifolds with vanishing 
first Chern class (see Proposition \ref{abelian}). Indeed, on a compact  K\"ahler manifold $M$ with vanishing first Chern 
class, any holomorphic projective connection admits a global representative which is a holomorphic torsionfree affine 
connection (see Lemma \ref{trivial class}). In this case it is known that $M$ admits a finite unramified cover 
which is a compact complex torus \cite{IKO} (the pull-back, to the torus, of
such a global representative affine
connection is a translation invariant holomorphic torsionfree affine connection). This type of results
are also 
valid in the broader context of holomorphic Cartan geometries \cite{BM1,BM2,Du2, BD4}; see
\cite{Sh} for holomorphic Cartan geometries.

Kobayashi and Ochiai proved in \cite{KO3} that holomorphic $G$--structures modeled on Hermitian symmetric spaces 
of rank $\,\geq\, 2$ on compact K\"ahler--Einstein manifolds are always flat (see also \cite{HM} for a similar result for  uniruled projective manifolds). The complex projective space being a 
Hermitian space of rank one, holomorphic projective connections constitute examples of holomorphic $G$--structures 
modeled on a Hermitian symmetric space of rank one. Theorem \ref{classification} may be seen as a rank one 
version of the above mentioned
result in \cite{KO3}. However, contrary to the situation in rank $\,\geq\, 2$, there exist non-flat 
holomorphic projective connections on compact complex tori of dimension three or more (see Proposition 
\ref{abelian}).

Complex projective threefolds admitting a holomorphic projective connection were classified by Jahnke and Radloff 
in \cite{JR1}. Their result says that any such projective threefold is
\begin{itemize}
\item either one among the standard ones (the complex projective space 
${\mathbb C}{\mathbb P}^3$, \'etale quotients of abelian threefolds and compact quotients of the complex hyperbolic 
$3$-space ${\rm \mathbb H}_{\mathbb C}^3$ by a torsion-free discrete subgroup in ${\rm SU}(3,1)$),

\item or an \'etale 
quotient of a Kuga--Shimura projective threefold (i.e., a smooth modular family of false elliptic curves; their
description is recalled in Section \ref{section 2}).
\end{itemize}
As noted in \cite{JR1}, each of these projective threefolds also admit a flat holomorphic projective connection.

We investigate the space of all holomorphic projective connections on Kuga--Shimura projective threefolds.
The main result in this direction is the following (proved in Section \ref{ptc}):

\begin{theorem}\label{main}
Let $M \,\longrightarrow\, \Sigma$ be a Kuga--Shimura projective threefold over a
Shimura curve $\Sigma$ of false elliptic curves. Then the following hold:
\begin{enumerate}
\item[(i)] The projective equivalence classes of holomorphic projective connections on $M$ are parametrized by
a complex affine space for the complex vector space $(\rm{H}^0(\Sigma,\, K_{\Sigma}^{\frac{3}{2}}))^2$. 

\item[(ii)] All holomorphic projective connections on $M$ are flat. The fibers of the Kuga--Shimura
fibration are totally geodesic with respect to every holomorphic projective connection on $M$.
\end{enumerate}
\end{theorem}

Theorem \ref{main} implies that the space of projective equivalence classes of flat holomorphic projective 
connections on Kuga--Shimura projective threefolds can have arbitrarily large dimension (see Remark \ref{KS}).

Theorem \ref{classification} and Theorem \ref{main}, combined with the 
classification in \cite{JR1}, give the following (proved in Section \ref{ptc}):

\begin{corollary}\label{cor Lich}
A holomorphic projective connection $\phi$ on a complex projective threefold is either flat, or
it is an \'etale quotient of a translation 
invariant holomorphic projective connection on an abelian threefold. In the second case, a generic translation
invariant holomorphic projective connection on an abelian variety of dimension three is not flat.
\end{corollary} 

Our motivation for Corollary \ref{cor Lich} comes from the projective Lichnerowicz conjecture. The projective 
Lichnerowicz conjecture roughly says that compact manifolds $M$ endowed with a projective connection $\phi$ 
admitting a connected (or, more generally, infinite) {\it essential}  group $G$ of automorphisms of $(M, \,\phi)$ 
(meaning, $G$ preserves $\phi$, but does not preserve any torsionfree affine connection representing $\phi$) are actually flat 
(i.e., $\phi$ is a flat projective connection). The literature on this subject is vast: see for instance \cite{Ma}, 
\cite{Ze} and references therein. In \cite{Ma}, the projective Lichnerowicz conjecture was solved in the Riemannian 
context (i.e., for the Levi--Civita connection of a Riemannian metric) and for connected  essential groups of 
projective automorphisms $G$; in \cite{Ze}, the same was proved for discrete infinite essential groups of 
projective automorphisms $G$. For local results in this direction, see, for instance, Theorem 3.1 in \cite{CM} 
which implies that analytic  projective connections admitting an essential local  projective Killing 
field are flat (compare this with \cite{NO}).
 
We formulate here a version of the projective Lichnerowicz conjecture for holomorphic pseudo-groups; there is no 
global transformation group in our formulation; instead we replace it with the pseudo-group of local 
biholomorphisms which are the transition maps of a compact complex manifold endowed with a holomorphic projective 
connection.

\begin{conjecture}\label{con1}
Let $M$ be a compact complex manifold bearing a holomorphic projective connection $\phi$. Assume that $M$ does 
not admit any global holomorphic torsionfree affine connection projectively equivalent to $\phi$. Then $\phi$ is flat.
\end{conjecture}

Lemma \ref{trivial class} shows that the assumption in
Conjecture \ref{con1} is equivalent to the assumption
that the canonical line bundle $K_M$ does not admit any holomorphic connection. Moreover, if $M$ is 
K\"ahler, this assumption is equivalent to the assumption that $c_1(M)\, \not=\, 0$.
So, in the K\"ahler setting, Conjecture \ref{con1} simplifies to the following:
 
\begin{conjecture}\label{con2}
Holomorphic projective connections on compact K\"ahler manifolds with nonzero first Chern class are flat.
\end{conjecture}

Theorem \ref{classification} gives a positive answer to Conjecture \ref{con2} for 
K\"ahler--Einstein manifolds, while Corollary \ref{cor Lich} gives a positive answer to 
Conjecture \ref{con2} for projective threefolds.
 
All simply connected K\"ahler manifolds, and, more generally, all simply connected manifolds in the 
Fujiki class $\mathcal C$ \cite{Fu} (i.e., compact complex manifolds bimeromorphic to a K\"ahler 
manifold \cite{Va}), bearing a holomorphic projective connection are actually complex projective 
manifold \cite[Theorem 4.3]{BD3}. In view of this, Corollary \ref{cor Lich} also gives a positive answer to 
Conjecture \ref{con1} for simply connected threefolds belonging to the Fujiki class $\mathcal C$. 
More precisely, a compact simply connected complex threefold in the Fujiki class $\mathcal C$ 
equipped with a holomorphic projective connection is isomorphic to ${\mathbb C}{\mathbb P}^3$ endowed 
with its standard flat projective connection.

For higher dimensions, a classification of complex projective manifolds admitting a flat holomorphic projective 
connection was obtained in \cite{JR2}. The classification of complex projective manifolds admitting a holomorphic 
projective connection (non necessarily flat) is still an open question. Notice that Conjecture \ref{con2} implies 
that compact K\"ahler manifolds bearing a holomorphic projective connection also admit flat holomorphic projective 
connections (all holomorphic projective connections are actually expected to be  flat, except the \'etale quotients of generic 
translation-invariant projective connections on compact complex tori).

An interesting class of compact non-K\"ahler threefolds with trivial canonical bundle and admitting a flat 
holomorphic projective connection is provided by parallelizable manifolds ${\rm SL}(2, \mathbb C) /\Gamma$, where 
$\Gamma$ is a cocompact lattice in ${\rm SL}(2, \mathbb C)$, along with the deformations of ${\rm SL}(2, \mathbb C) 
/\Gamma$ constructed in \cite{Gh} (which are, in general, not parallelizable manifolds). The details about the 
geometry of these projective connections can be found in {\cite{Gh} and \cite[Section 5]{BD}. It should be 
mentioned that compact complex non-K\"ahler parallelizable manifolds admitting a holomorphic projective connection, 
but not admitting any flat holomorphic projective connection, were constructed in \cite[Proposition 5.7]{BD}.

Section \ref{section 2} provides an introduction to the geometry of holomorphic projective 
connections as well as presentations of the standard models and the Kuga--Shimura threefolds. In 
Section \ref{section 3} we study holomorphic projective connections on K\"ahler--Einstein manifolds.
Section \ref{section 4} is about holomorphic projective 
connections on Kuga--Shimura manifolds, and contains proofs of Theorem \ref{main} and Corollary \ref{cor Lich}.

 Section \ref{section 5} deals with the compact non-K\"ahler threefolds, and the following theorem is proved
there (see Theorem \ref{thm sc}).

\begin{theorem}\label{thml}
A simply connected compact complex threefold with trivial canonical line bundle 
does not admit any holomorphic projective connection
\end{theorem}

A key ingredient in the proof of Theorem \ref{thml} is the result that the projective  Killing Lie algebra of a 
holomorphic projective connection on a compact complex threefold is nontrivial (see Proposition 
\ref{Killing}). We think that the statement proved in  Theorem \ref{thml} is likely to be true in higher dimensions. Our
Proposition \ref{Killing2} is a step in that direction.

More generally, we conjecture that a simply connected compact complex manifold bearing a holomorphic projective 
connection is isomorphic to the complex projective space (endowed with its standard flat structure); this is a 
version of Conjecture \ref{con1} for simply connected manifolds. We have seen above that this conjecture is verified 
for complex threefolds in Fujiki class $\mathcal C$, while Theorem \ref{thml} verifies it  for complex threefolds 
with trivial canonical bundle.

\section{Holomorphic projective connections}\label{section 2}

Recall that using the standard action of ${\rm PGL}(n+1, \mathbb C)$ on the complex projective space
${\mathbb C}{\mathbb P}^n$, the group of holomorphic automorphisms of
${\mathbb C}{\mathbb P}^n$ is identified with ${\rm PGL}(n+1, \mathbb C)$.

Let $M$ be a complex manifold of complex dimension $n$. A holomorphic coordinate function on $M$ is
a pair of the form $(U,\, \phi)$, where $U\, \subset\, M$ is an open subset and $\phi\, :\,
U\, \longrightarrow\, {\mathbb C}{\mathbb P}^n$ is a holomorphic embedding.
A {\it holomorphic projective 
structure}  on $M$ is given by a collection
of holomorphic coordinate functions $(U_i,\, \phi_i)_{i \in I}$ such that
\begin{itemize}
\item $\bigcup_{i \in I} U_i\,=\, M$, and

\item for $i,\, j\, \in\, I$, and each connected component $U_{ij,c}\, \subset\, U_i \cap U_j$, the transition
function $$\phi_i \circ \phi_j^{-1} \,: \, \phi_j(U_{ij,c}) \,\longrightarrow\, \phi_i(U_{ij,c})$$
coincides with the restriction of some $\phi_{ij,c} \,\in\, {\rm PGL}(n+1, \mathbb C)$.
\end{itemize}

An important consequence of the uniformization theorem for Riemann surfaces is that any Riemann surface admits 
a holomorphic projective structure \cite{Gu}.
In higher dimension the situation is much more stringent. All compact K\"ahler--Einstein manifolds admitting 
a holomorphic projective structure actually lie in one of the three standard examples described below \cite{KO,KO1}.

\subsection{The standard examples}

The complex projective space ${\mathbb C}{\mathbb P}^n$ is endowed with its standard holomorphic projective 
structure. This first of the three standard examples
is the model for any holomorphic projective structure in the following sense.

If $M$ is a complex simply connected manifold of complex dimension $n$, any holomorphic projective structure 
on $M$ is given by a holomorphic submersion (equivalently, immersion)
$${\rm dev}\,:\, M \,\longrightarrow\,
{\mathbb C}{\mathbb P}^n$$
which is known as the developing map. In 
particular, if $M$ is also compact, this ${\rm dev}$ is a covering map
and hence it is a biholomorphism, because ${\mathbb C}{\mathbb P}^n$ is simply 
connected. Therefore, the only compact simply connected complex $n$-manifold endowed with a 
holomorphic projective structure is ${\mathbb C}{\mathbb P}^n$ equipped with its standard projective structure.

Assume now that $M$ is endowed with a holomorphic projective structure $\phi$, but it is not simply connected 
anymore. Fix a point $x_0\, \in\, M$,
denote by $\varpi\, :\widetilde{M}\, \longrightarrow\, M$ the corresponding universal covering of $M$, and pull-back the holomorphic
projective structure $\phi$ to $\widetilde{M}$. Consider the developing map ${\rm dev}\,:\, \widetilde{M}
\,\longrightarrow\,{\mathbb C}{\mathbb P}^n$ for $\varpi^*\phi$. We have a unique homomorphism
$$
\rho\, :\, \pi_1(M,\, x_0)\, \longrightarrow\, {\rm PGL}(n+1, \mathbb C)
$$
such that ${\rm dev}$ is $\pi_1(M,\, x_0)$--equivariant for the natural action of
$\pi_1(M,\, x_0)$ on $\widetilde{M}$ and the action of $\pi_1(M,\, x_0)$ on
${\mathbb C}{\mathbb P}^n$ given by $\rho$ together with the standard action of ${\rm PGL}(n+1, \mathbb C)$
on ${\mathbb C}{\mathbb P}^n$; this $\rho$ is called the {\it monodromy homomorphism} for the
projective structure.

The projective structure on ${\mathbb C}{\mathbb P}^n$ induces a projective structure on every open 
subset of it. Take an open set $\Omega \,\subset\, {\mathbb C}{\mathbb P}^n$ and a discrete 
subgroup $\Gamma \,\subset\, {\rm PGL}(n+1, \mathbb C)$ that preserves $\Omega$ while acting freely 
and properly discontinuously on $\Omega$. Then the quotient complex manifold $\Omega / \Gamma$ 
inherits a holomorphic projective structure induced by that of $\Omega$.

The remaining two standard examples will be described by choosing appropriately the pair $(\Omega,\, \Gamma)$. 

\textit{Complex affine space and its quotients.}\, Take $\Omega$ to be the standard open affine subset
${\mathbb C}^n\, \subset\, {\mathbb C}{\mathbb P}^n$. 
Let $\Lambda \,\simeq\,{ \mathbb Z}^{2n}$ be some lattice in ${\mathbb R}^{2n}$; it acts on 
${\mathbb C}^n$ by translations. Since all affine transformations of ${\mathbb C}^n$ are restrictions of 
projective transformations, the complex compact torus ${\mathbb C}^n / \Lambda$ inherits a holomorphic 
projective structure.

\textit{Complex hyperbolic space and its quotients.}\, Let us now consider the complex hyperbolic space ${\rm 
\mathbb H}_{\mathbb C}^n$ of complex dimension $n$, seen as the Hermitian symmetric space ${\rm SU}(n,1)/{\rm S} 
({\rm U}(n,1) \times {\rm U}(1))$. The group ${\rm SU}(n,1)$ coincides with the group of holomorphic isometries of 
${\rm \mathbb H}_{\mathbb C}^n$. The compact dual of ${\rm \mathbb H}_{\mathbb C}^n$ is ${\mathbb C}{\mathbb P}^n$ 
acted on by the holomorphic isometry group ${\rm PU}(n+1)$ for its standard Fubini--Study K\"ahler metric. There is 
a canonical (Borel) embedding of ${\rm \mathbb H}_{\mathbb C}^n$ as an open subset of its compact dual ${\mathbb 
C}{\mathbb P}^n$. The image of this Borel embedding is the following ball in ${\mathbb C}{\mathbb P}^n$:
\begin{equation}\label{s1}
{\rm \mathbb H}_{\mathbb C}^n\, :=\, \{[Z_0 : Z_1 : \cdots : Z_{n}] \, \mid\, |Z_0|^2+ |Z_1|^2+ \ldots +
|Z_{n-1}|^2\,<\, |Z_n|^2\}\, \subset\,{\mathbb C}{\mathbb P}^n\, .
\end{equation}
The action of ${\rm SU}(n,1)$ on ${\rm \mathbb H}_{\mathbb C}^n$ evidently extends to an action of
${\rm SU}(n,1)$ on ${\mathbb C}{\mathbb P}^n$ by projective transformations. Therefore any quotient of
${\rm \mathbb H}_{\mathbb C}^n$ by a torsion-free discrete subgroup in ${\rm SU}(n,1)$ is a complex
manifold endowed with a holomorphic projective structure induced by the natural
holomorphic projective structure on the open subset
${\rm \mathbb H}_{\mathbb C}^n\, \subset\, {\mathbb C}{\mathbb P}^n$ in \eqref{s1}.

\subsection{False elliptic curves and Shimura curve}\label{ss Shimura} 

The main Theorem in \cite{JR1} asserts that any complex projective threefold bearing a holomorphic projective structure (or, more generally, a  holomorphic 
projective connection in the sense of Section  \ref{section  projective connection}) is
\begin{itemize}
\item either one of the above (three-dimensional) standard examples,

\item or an \'etale quotient of a smooth modular family of false elliptic curves.
\end{itemize}
We present here a construction of these compact Shimura curves of false elliptic curves, following 
the description in \cite{LB} and \cite{JR1}.

Let $\mathcal B$ be a totally indefinite quaternion algebra over $\mathbb Q$. More precisely, 
$\mathcal B$ is the algebra generated by two elements $i,j $ such that $$ij\,=\,-ji,\ \ 
i^2\,=\,a,\ \ j^2\, =\, b$$ for some $a,\,b\,\in\, \mathbb Q$ which are not both negative. Then $\mathcal B$ is a division 
algebra, and $$\mathcal B \otimes_{\mathbb Q} {\mathbb R}\,\simeq\, M_{2,2}(\mathbb R)\, .$$ 
Therefore elements of $\mathcal B$ may be seen as $(2,\,2)$-matrices with coefficients in some real 
quadratic number field.

A {\it false elliptic curve} is an abelian surface $T$ such that $\rm{End}(T)\otimes {\mathbb Q} \,\simeq\, 
{\mathcal B}.$

Let $\Lambda\,\simeq\,{ \mathbb Z}^4$ be some lattice in $\mathcal B$, and choose
a nontrivial anti-symmetric matrix
$$
M\,=\,
\begin{pmatrix}
0 & \alpha\\
- \alpha & 0
\end{pmatrix}\, \in \, \mathcal B$$
such that $\rm{tr} (\Lambda M \Lambda^t) \,\subset\, \mathbb Z$.
Denote by $\mathcal H$ the upper-half of complex plane.
For any $\tau \,\in\, \mathcal H$, we construct a complex structure on
$\mathcal B \otimes_{\mathbb Q} {\mathbb R}$ through the $\mathbb R$-vector space
isomorphism $$j_{\tau} \,:\, \mathcal B \otimes_{\mathbb Q} {\mathbb R} \,
\longrightarrow\, {\mathbb C}^2$$ defined as $A \,\longmapsto\,
A \cdot \begin{pmatrix}
\tau \\ 
\alpha
\end{pmatrix}$. 
There is a free and proper discontinuous action of $\Lambda$ on $\mathcal H \times \mathbb C^{2}$ given by 
$$ \lambda \cdot (h, (z_1, z_2))\,=\, (h, (z_1, z_2) + j_{\tau} (\lambda))$$ for any $\lambda\,\in\, \Lambda$.
The quotient for this action is a smooth nontrivial family $\Xi_{\mathcal B}$ of abelian
surfaces over $\mathcal H$.

Denote by $\overline{\Gamma}$ the stabilizer of the lattice $$\Lambda \begin{pmatrix} 1 & 0 \\
0 & \alpha\end{pmatrix}\, \subset\, {\rm SL}(2, \mathbb R)\, ,$$
and choose a torsionfree finite index subgroup $\Gamma \,\subset\, \overline{\Gamma}.$
Any element
$$\gamma \,=\,\begin{pmatrix} a & b \\
c & d \end{pmatrix}\,\in\, \Gamma
$$ 
acts on $\mathcal H$ by the conformal map $\tau \,\longmapsto\,\frac{a\tau +b}{c \tau +d}$.
We note that the fiber of $\Xi_{\mathcal B}$ over $\tau$ is
isomorphic to the fiber of $\Xi_{\mathcal B}$ over $\frac{a\tau +b}{c \tau +d}$ through the
multiplication by $\frac{1}{c \tau +d}$.

It follows that there is an action of the semi-direct product $\Gamma \ltimes \Lambda$ on 
$\mathcal H \times {\mathbb C}^2$ given by the map
\begin{equation}\label{s2}
(\gamma,\, \lambda) \cdot (\tau,\, z_1,\, z_2)\,\longmapsto\,
\left(\frac{a\tau +b}{c \tau +d},\, \frac{z_1+ m \tau +n}{c \tau +d},\,
\frac{z_2 + k \tau +l}{c \tau +d} \right)\, ,
\end{equation}
for all $\gamma\,= \,\begin{pmatrix} a & b \\
c & d\end{pmatrix}\,\in\,\Gamma$ and $\lambda\,=\,(m,\,n,\,k,\,l)\,\in\, \Lambda$. The
quotient of $\mathcal H \times {\mathbb C }^2$ by this action of $\Gamma \ltimes \Lambda$
is a projective abelian fibration
\begin{equation}\label{ksf}
\Xi_{\mathcal B} \,\longrightarrow\, \Sigma
\end{equation}
with base
$\Sigma \,=\, \mathcal H / \Gamma$ a compact Riemann surface of genus
$g \,\geq\, 2$; since $\mathcal B$ is a division algebra,
$\Gamma$ is a Fuchsian group such that $\mathcal H / \Gamma$ is compact \cite{Shi}.

A projective abelian fibration of the above type is called a {\it Kuga--Shimura projective threefold.}

Considering $\mathcal H \times \mathbb C^{2}$ as an open subset of ${\mathbb C}{\mathbb P}^3$, the action of 
$\Gamma\ltimes\Lambda$ on $\mathcal H \times \mathbb C^{2}$ in \eqref{s2} is evidently given by projective 
transformations. In particular, as it was first observed in \cite{JR1}, a Kuga--Shimura projective threefold is 
endowed with a flat holomorphic projective connection.

\subsection{Holomorphic projective connections and Weyl projective tensor}.  \label{section  projective connection}

Let $\mathcal Z$ be a complex manifold of complex dimension $n\,>\,1$.
A holomorphic connection on the holomorphic tangent bundle $T{\mathcal Z}$ of $\mathcal Z$
is called
a holomorphic affine connection on $\mathcal Z$ (see \cite{At} for holomorphic connection). A
holomorphic affine connection $\nabla$ on $\mathcal Z$ is called torsionfree if
$$\nabla_XY-\nabla_YX \,= \,\lbrack X,\, Y\rbrack$$ for all 
locally defined holomorphic vector fields $X$ and $Y$ on $\mathcal Z$.
Two holomorphic torsionfree affine connections $\nabla^1$ and $\nabla^2$ on $\mathcal Z$ are called {\it projectively 
equivalent} if they have the same non-parametrized holomorphic geodesics. This condition is
equivalent to the condition that there is a holomorphic $1$--form $\theta\,\in \,{\rm H}^0({\mathcal Z},\, T^*{\mathcal Z})$
such that
\begin{equation}\label{x1}
{\nabla^1}_XY\,=\,{\nabla^2}_XY+ \theta(X)Y+ \theta(Y)X
\end{equation}
for any locally
defined holomorphic vector fields $X,\, Y$ on $\mathcal Z$ (see \cite[p.~3021, Theorem 4.2]{MM},
\cite[p.~222, Proposition A.3.2]{OT}).

Let $M$ be a complex manifold of dimension $n\,>\,1$. A {\it holomorphic projective connection} on $M$ is given by 
a collection $(U_i,\, \nabla_i)_{i \in I}$, where
\begin{itemize}
\item $U_i\, \subset\, M$, $i\, \in\, I$, are open subsets with $\bigcup_{i\in I} U_i\, =\, M$, and

\item $\nabla_i$ is a torsionfree affine connection on $U_i$ such that for all $i,\, j\, \in\, I$, the
two affine connection $\nabla_i\vert_{U_i\cap U_j}$ and $\nabla_j\vert_{U_i\cap U_j}$ on $U_i\cap U_j$ are
projectively equivalent
\end{itemize}
(compare this with the equivalent definitions in \cite{KO} and \cite{MM}).
We say that the affine connection $(U_i, \, \nabla_i)$ is a local representative of the 
holomorphic projective connection.

Two holomorphic projective connections $(U_i, \,\nabla_i)_{i \in I}$ and $(U'_j, \,\nabla'_j)_{j \in J}$ 
are called \textit{projectively equivalent} if their union $\{(U_i, \,\nabla_i)_{i \in I},\,
(U'_j, \,\nabla'_j)_{j \in J}\}$ is again a holomorphic projective connection.

The above definition coincides with Definition 4.4 in \cite{MM}  and also with the definition given in \cite[Chapter 8]{Gu}  (see the proof in \cite{MM}  showing that the two definitions are equivalent). It should be 
mentioned that some authors call those projective connections, which are locally represented by {\it torsionfree} 
affine connections, as {\it normal} \cite{Ka,KO,JR1}. In their terminology we work, throughout the article, with {\it 
holomorphic normal projective connections.}

A holomorphic projective connection is called {\it flat} if it is projectively equivalent to a
holomorphic projective connection $(U_i, \,\nabla_i)_{i \in I}$, where each $\nabla_i$ is flat. This means that
a suitable holomorphic coordinate function on $U_i$ takes $\nabla_i$ to the standard connection
on ${\mathbb C}^n$. Once we fix holomorphic coordinate functions on every $U_i$ satisfying
the above condition that it takes $\nabla_i$ to the standard connection on ${\mathbb C}^n$,
the transition functions defined on the intersections $U_i \cap U_j$ are projective transformations between 
open subsets of ${\mathbb C}{\mathbb P}^n$. Hence
manifolds endowed with a flat holomorphic projective connection are locally modeled
on the complex projective space. Consequently, flat holomorphic projective connections on $M$ are
precisely the holomorphic projective structures on 
$M$.

The curvature tensor of a holomorphic affine connection $\nabla$ on $M$ is defined to be 
\begin{equation}\label{ct}
R(X,\,Y)Z\,=\,\nabla_X \nabla_YZ- \nabla_Y \nabla_X Z- \nabla_{\lbrack X, Y \rbrack } Z\, ,
\end{equation}
where $X,\,Y,\,Z$ are locally defined holomorphic vector fields on $M$. So $$R\, \in\, {\rm H}^0(M,\, 
(\bigwedge\nolimits^2 T^*M)\otimes {\rm End}(TM))\, .$$
The curvature $R$ vanishes identically if and only if 
$\nabla$ is locally isomorphic to the standard affine connection of $\mathbb C^n$. Let
${\rm End}^0(TM)\, \subset\, {\rm End}(TM)$  be  the direct summand given by the
endomorphisms of trace zero of the fibers. The trace-free
part of $R$, which is a holomorphic section of $(\bigwedge^2T^*M)\otimes {\rm End}^0(TM)$,
is the \textit{Weyl projective curvature} of $\nabla$. While the curvature $R$ is 
not a projective invariant, its Weyl projective curvature $W$ is evidently a projective invariant.

The Weyl projective tensor of a holomorphic projective connection on $M$ is the Weyl projective curvature
of the local representatives $(U_i, \,\nabla_i)_{i \in I}$ of the holomorphic projective connection.
A holomorphic projective connection is flat if and only if the associated Weyl projective tensor vanishes 
identically (see \cite{Ei,We} and \cite[p.~79--83]{Gu} for the adaptation of the proof to the holomorphic 
setting).

The expression of the Weyl curvature in dimension three --- the case of interest in this article --- is 
the following (see, for example, formula (3.4) on \cite[p.~114]{Ga}):
\begin{equation}\label{wt}
W(X,\,Y)Z\,= \,R(X,\,Y)Z-\frac{1}{4} \rm{Tr}R(X,\,Y)Z-\frac{1}{2}(\rm{Ricci}(Y,\,Z)X
\end{equation}
$$
-\rm{Ricci}(X,\,Z)(Y))
- \frac{1}{8} (\rm{Tr}R(Y,\,Z)X-\rm{Tr}R(X,\,Z)(Y))\, .
$$
In the expression in \eqref{wt},
\begin{equation}\label{ct2}
\rm{Ricci}\, \in\, {\rm H}^0(M,\, T^*M^{\otimes 2})
\end{equation}
is the Ricci curvature that sends
$\eta \otimes  \nu \in T_xM^{\otimes 2}$ to the trace of the endomorphism of $T_xM$ defined by
$\xi\, \longmapsto\, R(x)(\xi, \,\eta) \nu$. Also, in \eqref{wt},
$$\rm{Tr}R\, \in\, {\rm H}^0(M,\, \bigwedge\nolimits^2T^*M)$$ sends
$\eta\wedge\nu\, \in\, \bigwedge^2T_xM$ to the trace of the endomorphism of $T_xM$ defined by
$\xi \, \longmapsto\, R(\eta, \,\nu) \xi$.

We have $\rm{Tr}R(X,\,Y) \,=\,\rm{Ricci} (Y,\,X) -\rm{Ricci}(X,\,Y)$, and hence the Weyl
tensor in \eqref{wt} can be expressed in terms of just the Ricci tensor (this is exactly the formula (3.4) on
\cite[p.~114]{Ga}):
$$W(X,\,Y)Z\,=\,R(X,\,Y)Z + \frac{1}{4}(\rm{Ricci}(X,Y)Z - \rm{Ricci}(Y,X)Z)
$$
$$
+ \frac{1}{8} \lbrack (3 \rm{Ricci} (X,Z) + \rm{Ricci} (Z,X))Y -( 3\rm{Ricci} (Y,Z) + \rm{Ricci} (Z,Y))X \rbrack \, .$$

We note that ${\rm Ricci}$ is  not a projective invariant, while, in contrast,  as mentioned before,  $W$
is projectively invariant.  Moreover,  the Weyl tensor $W$ 
possesses the same tensorial symmetries as $R$.
In particular, $W\, \in\, {\rm H}^0(M,\, \bigwedge^2T^*M)\otimes {\rm End}^0(TM))$
satisfies the first Bianchi identity which says that 
\begin{equation}\label{bii}
W(X,\,Y)Z + W(Y,\,Z)X + W(Z,\,X)Y\,=\,0
\end{equation}
for all locally defined holomorphic vector fields $X,\,Y,\,Z$ on $M$.

The local symmetries for a holomorphic projective connection are given by the local {\it (projective) Killing 
field.} For a holomorphic projective connection $\phi$ on $M$, 
a local holomorphic vector field $\mathbb K$ on $M$ is a local {\it projective Killing field} 
(or briefly Killing field, when there is no ambiguity) if the local flow for $\mathbb K$ preserves
$\phi$. When the local flow for $\mathbb K$ preserves a holomorphic affine 
connection $\nabla$ representing $\phi$, then $\mathbb K$ is called a local {\it affine Killing 
field}.

The local Lie algebra formed by all local projective Killing fields has finite dimension. The dimension of the 
projective Killing Lie algebra for $(M,\, \phi)$ is at most $(n+1)^2-1$, where $n\,=\, \dim_{\mathbb C} M$. This 
maximal bound is realized only for projectively flat manifolds: in this case the local projective Killing Lie 
algebra is isomorphic to the Lie algebra of ${\rm PGL}(n+1, \mathbb C)$.

\section{Holomorphic projective connections on K\"ahler-Einstein manifolds} \label{section 3}

In this Section we study holomorphic projective connections on compact K\"ahler--Einstein manifolds 
and prove Theorem \ref{classification}. 

According to \cite{KO,KO1}, the only compact K\"ahler--Einstein manifolds of dimension $n$ admitting a holomorphic 
projective connection are the standard ones: the complex projective space ${\mathbb C}{\mathbb P}^n$, the compact 
quotients of the complex hyperbolic $n$-space ${\rm \mathbb H}_{\mathbb C}^n$ by a torsion-free discrete subgroup 
in ${\rm SU}(n,1)$ and the \'etale quotients of compact complex $n$-tori.

The case of holomorphic projective connections on quotients of ${\rm \mathbb H}_{\mathbb C}^n$ will be settled in 
Corollary \ref{cor1}. The case of the complex projective space ${\mathbb C}{\mathbb P}^n$ will be settled in 
Corollary \ref{cor2} (more general results are known from \cite{Ye,JR1,BM}).
Both of  these results are direct consequences of Lemma \ref{lemma} which parametrizes the space of projective 
classes of holomorphic projective connections on a complex manifold (compare this with \cite[Proposition 5.7]{Kl1} 
and \cite[Proposition 2.1]{MY} for the flat case).

Holomorphic projective connections on compact complex tori 
are studied in Proposition \ref{abelian}. 

Let us first state a technical result which will be useful in the sequel (compare it with 
{\cite[p.~7449, Lemma 5.6]{BD} where the sufficient condition was proved for compact 
manifolds with trivial canonical bundle).

\begin{lemma}\label{trivial class}
Let $\phi$ be a holomorphic projective connection on a complex manifold $M$. Then $M$ admits a holomorphic 
torsionfree affine connection $\nabla$ which is projectively equivalent to $\phi$ if and only if the canonical line 
bundle $K_M$ admits a holomorphic connection. If $M$ is compact and K\"ahler, this condition is equivalent with the 
condition that $c_1(M)\,=\, 0$.
In particular, the above condition is automatically satisfied if $K_M$ is trivial.
\end{lemma}

\begin{proof}
The proof is obtained as a direct consequence of the results in \cite{KO} (see also \cite{Gu}). There exists a 
holomorphic affine connection representing the projective connection $\phi$ if and only if the cocycle (3.2) in 
\cite{KO} defined as $d\log\Delta_{ij}$, where $\Delta_{ij}$ is the 1-cocycle of the canonical bundle $K_M$, 
vanishes in the cohomology group ${\rm H}^1(M,\, \Omega_M)$ of the sheaf of holomorphic one-forms (see the explicit 
formula (3.6) on \cite[p.~78--79]{KO}). This vanishing condition is satisfied if and only if the canonical line 
bundle $K_M$ admits a holomorphic connection (it coincides with the condition that the Atiyah class for $K_M$ 
vanishes \cite[Theorem 5, p.~195]{At}; see also \cite[p.~96--97]{Gu} for an alternative approach). For compact 
K\"ahler manifolds, this condition is equivalent to the condition that $c_1(M)\,=\, 0$ \cite[Proposition 12, 
p.~196]{At}.

If $K_M\,=\, {\mathcal O}_M$, the existence of a (global) holomorphic torsionfree affine connection representing 
$\phi$ was proved in \cite[p.~7449, Lemma 5.6]{BD}. For the convenience of the reader, we include here a short 
proof which will be needed in the proof of Proposition \ref{Killing2}.

Let $M\,= \,\bigcup_{i \in I} U_i$ be an open cover of $M$ such that on each $U_i$ there is a holomorphic 
torsionfree affine connection $\nabla_i$ projectively equivalent to the given projective connection $\phi$. In 
particular, $\nabla_i$ and $\nabla_j$ are projectively equivalent on $U_i\bigcap U_j$. Let $\omega$ be a 
holomorphic volume form on $M$ (i.e., $\omega$ is a trivializing holomorphic section of $K_M$). On 
each $U_i$, there exists a unique holomorphic torsionfree affine connection $\widetilde{\nabla}_i$ projectively 
equivalent to $\nabla_i$ such that $\omega$ is parallel with respect to $\widetilde{\nabla}_i$ \cite[Appendix 
A.3]{OT}. By uniqueness, $\widetilde{\nabla}_i$ and $\widetilde{\nabla}_j$
agree $U_i\bigcap U_j$ for all $i,\, j\, \in\, I$. Consequently, the connections $\{\widetilde{\nabla}_i\}$
together define a global 
holomorphic torsionfree affine connection on $M$ which is projectively equivalent to $\phi$.
\end{proof}

The $i$--fold symmetric product of a vector bundle $V$ would be denoted by $S^i(V)$.
For a complex manifold $M$, let
\begin{equation}\label{x3}
{\rm div}\,:\, S^2(T^*M) \otimes TM\, \longrightarrow\, T^*M
\end{equation}
be the map constructed by combining the natural homomorphism $$S^2(T^*M) \otimes TM
\,\longrightarrow\, T^*M
\otimes {\rm End}(TM)$$ with the trace map ${\rm Tr} : {\rm End}(TM)
\,\longrightarrow\, {\mathcal O}_M$. The resulting map in \eqref{x3} is denoted
by ${\rm div}$ (not to be mixed with the earlier map ${\rm dev}$)
because it can be seen as a divergence operator defined on
the space of quadratic vector fields (see \cite[p.~180]{OT}). Now define
\begin{equation}\label{et0}
(S^2(T^*M) \otimes TM)_0\, :=\, \text{kernel}({\rm div})\, \subset\, S^2(T^*M) \otimes TM\, .
\end{equation}

A section of  $(S^2(T^*M) \otimes TM)_0$ will be called  {\it  trace-free.}

The next lemma generalizes to the non-flat case a known result for flat projective 
connections; compare it with \cite[Proposition 5.7]{Kl1} and \cite[Proposition 2.1]{MY}, and 
notice that the bundle $\pi_{*}Hom(L,S)$ in \cite[Proposition 2.1]{MY} is isomorphic to 
$(S^2(T^*M) \otimes TM)_0$ defined in \eqref{et0}.

\begin{lemma} \label{lemma}
Let $M$ be a complex manifold of complex dimension $n \,>\,1$ endowed with a
holomorphic projective connection. Then the space of projective equivalence classes of
holomorphic projective connections on $M$ is identified with
${\rm H}^0(M,\, (S^2(T^*M) \otimes TM)_0)$ (see \eqref{et0}).
\end{lemma}

\begin{proof}
Fix a point $x_0\, \in\, M$, and let $$\varpi\, :\, \widetilde{M}\, \longrightarrow\, M$$
be the corresponding universal cover of $M$.
Let $\phi$ be holomorphic
projective connection on $M$. Let $\varpi^*\phi$ be the holomorphic
projective connection on $\widetilde{M}$ obtained by pulling back $\phi$.

First assume that the canonical bundle of $\widetilde{M}$ is holomorphically trivial.
Now Lemma \ref{trivial class} implies that $\varpi^*\phi$ is represented by a
(globally defined) torsionfree holomorphic affine connection. Let $\nabla^0$ be
such a global representative.

Now consider another holomorphic projective connections $\phi'$ on $M$. Let $\nabla$
be a holomorphic affine connection on $\widetilde{M}$ that represents
$\varpi^*\phi'$. Then
\begin{equation}\label{x2}
\Theta\, := \,\nabla - \nabla^0 \,\in\, {\rm H}^0(\widetilde{M},\, 
S^2(T^*\widetilde{M}) \otimes T\widetilde{M})\, .
\end{equation}
It should be mentioned that $\Theta$ lies in the subspace
$$
{\rm H}^0(\widetilde{M},\, S^2(T^*\widetilde{M}) \otimes T\widetilde{M})\, \subset\,
{\rm H}^0(\widetilde{M},\, (T^*\widetilde{M})^{\otimes 2} \otimes T\widetilde{M})
$$
because both $\nabla$ and $\nabla^0$ are torsionfree.

The natural action of $\pi_1(M,\, x_0)$ on $\widetilde{M}$ is evidently by $\nabla^0$--projective 
transformations. Although this action does not preserve the connection $\nabla^0$, the 
action of any element $\gamma \,\in \,\pi_1(M)$ does send $\nabla^0$ to an affine connection 
\begin{equation}\label{ng}
\nabla_{\gamma}
\end{equation}
which is projectively equivalent to $\nabla^0$, meaning there is a holomorphic one-form 
$\phi_{\gamma} \,\in\, {\rm H}^0(\widetilde{M},\, T^*\widetilde{M})$ such that 
$$({\nabla_{\gamma}})_XY \,=\, \nabla^0_X Y+ \phi_{\gamma}(X)Y + \phi_{\gamma}(Y)X$$ (see 
\eqref{x1}). If $\Theta$ in \eqref{x2} is invariant under the action of $\pi_1(M,\, x_0)$ on 
$\widetilde M$, then $\gamma \,\in\, \pi_1(M,\, x_0)$ sends $\nabla\,=\,\nabla^0 + \Theta$ 
to the projectively equivalent connection $\nabla_{\gamma} + \Theta$, where $\nabla_\gamma$ 
is the connection in \eqref{ng}. This immediately implies that the action of $\pi_1(M,\, 
x_0)$ on $\widetilde M$ does factor through the $\nabla$-projective transformations. 
Consequently, $\nabla$ descends to $M$ as a holomorphic projective connection.

Therefore, the $\pi_1(M,\, x_0)$--invariance of $\Theta$ is a sufficient condition for
$\nabla^0 + \Theta$ to descend to $M$ as a holomorphic projective connection.

We shall prove that the trace-free part of $\Theta$ is $\pi_1(M,\, x_0)$--invariant if and only if $\nabla$
descends to $M$ as a holomorphic projective connection.

There is a natural injection
\begin{equation}\label{cj}
{\mathcal J} \,: \,T^*\widetilde{M}\,\longrightarrow\,
S^2(T^*\widetilde{M}) \otimes T\widetilde{M}
\end{equation}
that sends any $l\, \in\, T^*_y\widetilde{M}$ to the
homomorphism ${\mathcal J}(l) \,:\, S^2(T_y\widetilde{M})\,\longrightarrow\,T_y\widetilde{M}$ defined by
$u\otimes v \,\longmapsto\, l(u) v+l(v) u$. 
It is straightforward to check that for ${\rm div}$ in \eqref{x3},
$${\rm div} \circ {\mathcal J}\,= \,(n+1) {\rm Id}\, ,$$
and hence the decomposition into a direct sum
$$
S^2(T^*\widetilde{M}) \otimes T\widetilde{M}\,=\,(S^2(T^*\widetilde{M}) \otimes T\widetilde{M})_0
\oplus {\rm Im}({\mathcal J})
$$
is obtained, where $(S^2(T^*\widetilde{M}) \otimes T\widetilde{M})_0\,=\,\text{kernel}({\rm div})$
(as in \eqref{et0}). The projection
\begin{equation}\label{T}
{\mathbb F}\, :\, S^2(T^*\widetilde{M}) \otimes T\widetilde{M}\,\longrightarrow\,
(S^2(T^*\widetilde{M}) \otimes T\widetilde{M})_0
\end{equation}
for the above decomposition
coincides with the map defined by $\Theta \,\longmapsto\, \Theta- \frac{1}{n+1} ({\mathcal J} \circ {\rm div})$,
where ${\mathcal J}$ is constructed in \eqref{cj} (see \cite[p.~180]{OT}).

Next from the definition of the projectively equivalent connections (compare with the expression of
${\mathcal J}$) it follows 
that the action of $\pi_1(M,\, x_0)$ on $\widetilde M$ is via
 projective equivalent maps with respect to $\nabla^0 + \Theta$ if and only if $\pi_1(M,\, x_0)$ 
preserves the trace-free part of $\Theta$. Therefore, the holomorphic affine connection $\nabla^0 + \Theta$ on 
$\widetilde{M}$ descends to a well-defined holomorphic projective connection on $M$ if and only if
$${\mathbb F}(\Theta)\,\in\, {\rm H}^0(\widetilde{M},\, (S^2(T^*\widetilde{M}) \otimes T\widetilde{M})_0)$$
is $\pi_1(M,\, x_0)$--invariant, where $\mathbb F$ is the projection in \eqref{T}. Consequently, the space 
of projective equivalence classes of holomorphic projective connections on $M$ is identified with the space of 
 holomorphic sections of $(S^2(T^*M) \otimes TM)_0$.

Now consider the general case where $M$ is any complex manifold endowed with a holomorphic projective connection
$\phi_0$. Let $(U_i, \, \nabla^0_i)_{i\in I}$ be a covering of $M$ by local representatives of $\phi_0$, where 
$\nabla^0_i$ is a holomorphic torsionfree affine connection on $U_i$
that represents $\phi_0\vert_{U_i}$. Of course on each intersection $U_i 
\bigcap U_j$ the connections $\nabla^0_i$ and $\nabla^0_j$ are projectively equivalent. On each open subset $U_i$,
the canonical line bundle $K_{U_i}\,=\, K_M\vert_{U_i}$ admits a holomorphic affine connection induced by $\nabla^0_i$.

Next take another holomorphic projective connection $\phi$ on $M$. By Lemma \ref{trivial class}, on each open 
subset $U_i$ there exists a holomorphic torsionfree affine connection $\nabla_i$ representing $\phi\vert_{U_i}$. Define 
$$\Theta_i \,:=\, \nabla_i-\nabla^0_i$$ on each $U_i$; it is a holomorphic section of $S^2(T^*M) \otimes TM$ over 
$U_i$. By the above considerations, the trace zero-part ${\mathbb F}(\Theta_i)$, where $\mathbb F$ is constructed 
in \eqref{T}, does not depend on the choices of the local representatives $\nabla^0_i$ and $\nabla_i$. This implies 
that the local sections ${\mathbb F}(\Theta_i)$ and ${\mathbb F}(\Theta_j)$ coincide on $U_i\bigcap U_j$. 
Consequently, the local sections ${\mathbb F}(\Theta_i)$ glue together compatibly to produce a global holomorphic 
section of $(S^2(T^*M) \otimes TM)_0$ (defined in \eqref{et0}) over $M$.

Conversely, take any $\Theta \,\in \, {\rm H}^0(M,\, (S^2(T^*M) \otimes TM)_0)$. On each open
subset $U_i$, consider 
the holomorphic affine connection $\nabla^0_i + \Theta_i$, where $\Theta_i$ is the restriction of $\Theta$ to 
$U_i$. Over $U_i\bigcap U_j$, the difference 
$(\nabla^0_i + \Theta_i)- (\nabla^0_j + \Theta_j)$
between the two holomorphic affine connections 
is a holomorphic section of $S^2(T^*M) \otimes TM$ over $U_i\bigcap U_j$ that lies in the
image of the homomorphism
${\mathcal J}$ in \eqref{cj}. This implies that the restrictions of $(\nabla^0_i + \Theta_i)\vert_{U_i\cap U_j}$ and 
$(\nabla^0_j + \Theta_j)\vert_{U_i\cap U_j}$ are projectively equivalent. Therefore, the collection
$(U_i,\, \nabla^0_i+\Theta_i)_{i\in I}$ defines a holomorphic projective connection on $M$, which
will be denoted by $\phi$.

By construction, the holomorphic projective connection $\phi$ constructed above is 
projectively equivalent to $\phi_0$ if and only if the corresponding section $\Theta$ vanishes identically.
\end{proof}

\begin{corollary}\label{cor1} 
Let $M$ be a  quotient of the complex hyperbolic space
${\rm \mathbb H}_{\mathbb C}^n$ by a torsion-free  lattice with finite covolume in
 ${\rm S}{\rm U}(n,1)$, with $n>1$. Then there is a unique holomorphic projective connection
on $M$, namely the standard flat one.
\end{corollary} 

\begin{proof}
It was proved in \cite[Proposition 4.10]{Kl1}, and earlier in \cite[Section 3]{MY} for the cocompact case, that
${\rm H}^0(M,\, (S^2(T^*M) \otimes TM)_0)\,=\, 0$. Therefore, $M$ has
at most one holomorphic projective connection by Lemma \ref{lemma}.
\end{proof} 

\begin{remark}
Note that Lemma \ref{lemma} does not hold for $n\,=\,1$. Indeed, it is classically known that the space of
holomorphic projective structures on a Riemann surface $\Sigma$ is an affine space for the vector space of
holomorphic quadratic differentials on $\Sigma$ (see \cite{Gu} or Chapter 8 in \cite{StG}). Lemma
\ref{lemma} is a higher dimensional version of this classical result.
Also, Corollary \ref{cor1} does not hold for $n\,=\,1$ for the same reason.
\end{remark}

\begin{corollary} \label{cor2}
The complex projective space ${\mathbb C}{\mathbb P}^n$ admits a unique holomorphic projective
connection, namely the standard flat one.
\end{corollary}

\begin{proof}
If $n\,=\,1$, a holomorphic projective connection is automatically flat. Since the complex projective line is simply 
connected, the developing map for a projective structure produces an
isomorphism of the projective structure with the standard projective structure of ${\mathbb 
C}{\mathbb P}^1$.

Now assume that $n\,>\,1$.
We will prove that
\begin{equation}\label{cz}
{\rm H}^0({\mathbb C}{\mathbb P}^n,\, S^2(T^*{\mathbb C}{\mathbb P}^n)\otimes
T{\mathbb C}{\mathbb P}^n)\,=\, 0\, .
\end{equation}
The Fubini--Study metric on ${\mathbb C}{\mathbb P}^n$ is
K\"ahler--Einstein. So the Hermitian structure on $S^2(T^*{\mathbb C}{\mathbb P}^n)\otimes T{\mathbb C}{\mathbb P}^n$
induced by the Fubini--Study metric is Hermitian--Einstein. On the other hand,
$$
{\rm degree}(S^2(T^*{\mathbb C}{\mathbb P}^n)\otimes T{\mathbb C}{\mathbb P}^n)\, <\, 0
$$
\cite[p.~50, Theorem 2.2.1]{LT}.
Hence \eqref{cz} holds by stability. Therefore, Lemma \ref{lemma} implies that ${\mathbb C}{\mathbb P}^n$ has
a unique holomorphic projective connection: it is the standard one.
\end{proof}

The following proposition (statement (ii)) studies holomorphic projective connections on compact complex tori.
Statement (i), which was already known in the broader context of holomorphic Cartan geometries
(see for example, \cite{BM1,BM2,BD4,Du2}), shows that compact complex tori cover the case of K\"ahler manifolds
with trivial first Chern class.

\begin{proposition}\label{abelian}\mbox{}
\begin{enumerate}
\item[(i)] Let $M$ be a compact K\"ahler manifold with $c_1(M)\,=\, 0$ bearing a holomorphic 
projective connection $\phi$. Then $M$ admits a finite unramified cover $T$ which is a
compact complex torus; the pull-back of $\phi$ on $T$ is projectively equivalent
to a (translation invariant) holomorphic torsionfree affine connection.

\item[(ii)] A generic holomorphic projective connection on a compact complex torus of complex dimension
$n\,>\,2$ is not 
projectively flat.
\end{enumerate}
\end{proposition}

\begin{proof} (i) By Calabi's conjecture proved by Yau, \cite{Ya}, $M$ admits a Ricci flat
K\"ahler metric. Using this, Bogomolov--Beauville decomposition theorem,\cite{Be,Bo1}, shows that
$M$ admits a finite unramified cover $\psi\, :\, M'\, \longrightarrow\, M$ such that the canonical
line bundle $K_{M'}$ is trivial. Now
Lemma \ref{trivial class} says that the holomorphic projective connection $\psi^*\phi$ is 
represented by a holomorphic affine connection on $M'$.
Since $M'$ is K\"ahler, it is known that $M'$ admits a finite unramified cover $T$ which is a complex torus
\cite{IKO}. Any holomorphic affine connection on $T$ is known to
be translation invariant \cite{IKO} (see also the proof below).

(ii) Let ${\mathbb T}^n$ be a compact complex torus of complex dimension $n\, >\,2$. Since the canonical bundle
$K_{{\mathbb T}^n}$ of the torus is 
trivial, Lemma \ref{trivial class} says that every holomorphic projective connection on ${\mathbb T}^n$
is represented by some globally defined holomorphic affine connection.

Denote by $(z_1,\, \cdots,\, z_n)$ a holomorphic linear coordinate function on ${\mathbb T}^n$ and by $\nabla_0$ the
standard flat holomorphic affine connection of ${\mathbb T}^n$ (induced by that of ${\mathbb C}^n$ using this
coordinate function). Any 
holomorphic affine connection on ${\mathbb T}^n$ is of the form $\nabla_0 + \Theta$, where
$\Theta\,\in\, {\rm H}^0({\mathbb T}^n,\, 
S^2(T^* {\mathbb T}^n) \otimes T{\mathbb T}^n)$. Since the holomorphic tangent bundle $T{\mathbb T}^n$ is
holomorphically trivial, such a 
section $\Theta$ is a sum of terms of the form $f_{ij}^k dz_idz_j \frac{\partial}{\partial z_k},$ where $f_{ij}^k$ 
are constant functions on ${\mathbb T}^n$ with values
in $\mathbb C$. The coefficients $f_{ij}^k $ are classically called the Christoffel symbols of the 
affine connection. In particular, the holomorphic affine connection $\nabla_0 + \Theta$ is translation-invariant.

First assume $n\,=\,3$. For the convenience of computations, let us denote by $(z_1,\,z_2,\, \tau)$ 
the holomorphic linear coordinate function on ${\mathbb T}^3$.

Denote by $\Theta_{A, B,C,D,E} $ the holomorphic section of $S^2(T^*{\mathbb T}^3) \otimes T{\mathbb T}^3$
corresponding to the following Christoffel symbols:
$$f_{\tau , \tau }^{z_1}\,=\,A,\ \ f_{\tau , \tau }^{z_2}\,=\,B\, ,$$
$$f_{z_1,z_1}^{z_1}\,=\,2f_{\tau ,z_1}^{\tau }\,=\,2f_{z_1,z_2}^{z_1}\,=\, C\, ,$$
$$f_{z_2 ,z_2}^{z_2}\,=\,2f_{\tau ,z_2}^{\tau }\,=\,2f_{z_1,z_2}^{z_2}\,=\,D\, ,$$
$$f_{\tau , \tau }^{\tau}\,=\,2f_{z_1, \tau}^{z_1 }\,=\,2f_{z_2, \tau }^{z_2} \,=\,E\, ,$$
with $A,\,B,\,C,\,D,\,E \,\in\, \mathbb C$; all other remaining coefficients are set to zero.

Consider the associated holomorphic affine connection $$\nabla^{A,B,C,D,E}\,=\, \nabla_0 +
\Theta_{A,B,C,D,E}\, .$$

We will prove in Appendix the  following  technical 

\begin{lemma} \label{lemma appendix} $\nabla^{A,B,C,D,E}$ is projectively flat on ${\mathbb T}^3$ if and only if $C\,=\,D$.

\end{lemma}

Equivalently, the Weyl projective tensor $W$ of    $\nabla^{A,B,C,D,E}$  vanishes identically on  ${\mathbb T}^3$  if and only if $C=D$.
In particular,  for generic $A,B,C,D,E$ the connection  $\nabla^{A,B,C,D,E}$    is not projectively flat on  ${\mathbb T}^3$.

Now consider the general connection on ${\mathbb T}^3$ given by $$\nabla\,=\,\nabla_0 + \Theta\, ,$$
where $\Theta \,\in\, {\rm H}^0({\mathbb T}^3,\, S^2(T^*{\mathbb T}^3)\otimes T{\mathbb T}^3)$. The
vanishing of the Weyl tensor
for $\nabla$ is an algebraic (quadratic) 
equation in the Christoffel symbols $f_{ij}^k$. Since all connections $\nabla^{A,B,C,D,E}$ with 
$C\,\neq\, D$ have nonzero Weyl tensor, the space of flat projective connections has positive codimension in the 
space of all connections. Consequently, the general connection is not flat.

Now consider the case where the complex dimension of the torus is $n\,>\, 3$. Denote by
$(z_1,\,z_2,\, \tau =z_3,\, z_4,\, \cdots,\, z_n)$ a global linear
holomorphic coordinate function on ${\mathbb T}^n$. Consider the holomorphic projective connection represented by
$\nabla_n^{A,B,C,D,E}\,=\, \nabla_0 + \Theta_n^{A,B,C,D,E}$, where $$\Theta_n^{A,B,C,D,E}
\,\in \,{\rm H}^0({\mathbb T}^n,\, S^2(T^*{\mathbb T}^n)\otimes T{\mathbb T}^n)$$
is defined below by the Christoffel symbols:
$$f_{\tau , \tau }^{z_1}\,=\, A, f_{\tau , \tau }^{z_2}\,=\, B\, ,$$
$$f_{z_1,z_1}^{z_1}\,=\,2f_{\tau ,z_1}^{\tau }=2f_{z_1,z_2}^{z_1}\,=\,C\, ,$$
$$f_{z_2 ,z_2}^{z_2}\,=\,2f_{\tau ,z_2}^{\tau }\,=\,2f_{z_1,z_2}^{z_2}\,=\,D\, ,$$
$$f_{\tau , \tau }^{\tau}\,=\,2f_{z_1, \tau}^{z_1 }\,=\,2f_{z_2, \tau }^{z_2} \,=\, E\, ,$$
where $A,\,B,\,C,\,D,\,E \,\in \,\mathbb C$; the remaining symbols are trivial.

Identify ${\mathbb C}^n$ with the universal cover of ${\mathbb T}^n$, and equip ${\mathbb C}^n$ with the
connection given by the connection $\nabla^{A,B,C,D,E}_n$ on ${\mathbb T}^n$ using this identification.
By construction, the three dimensional linear subspace
$$
\{(z_1,\, \cdots ,\, z_n)\, \in \, {\mathbb C}^n\, \mid\, z_4\,=\, \ldots\, =\,z_n\,=\,0\}
\, \subset\, {\mathbb C}^n
$$
is totally geodesic, and the induced connection on this subspace
is the connection $\nabla^{A,B,C,D,E}$ studied earlier. For 
$C\,\neq\, D$, the connection $\nabla^{A,B,C,D,E}$ is not projectively flat, and hence
$\nabla_n^{A,B,C,D,E}$ is not projectively flat either.

The same argument as in the dimension three
case proves that the generic connection $\nabla$ on 
${\mathbb T}^n$ (meaning $\nabla$ lying in a Zariski dense open set, whose complement of a proper algebraic 
subvariety defined by quadratic equations) is not flat.
\end{proof}

Ye proved in \cite{Ye} that Fano manifolds bearing a holomorphic projective connection are flat, 
isomorphic to the standard complex projective space. This was generalized in \cite{JR1} to compact K\"ahler 
manifolds admitting nontrivial rational curves. More recently this result was extended to the
more general context of 
holomorphic Cartan geometries (see Theorem 2 in \cite{BM}).

Moreover, Hwang and Mok proved in \cite{HM} (Theorem 2 and Proposition 8) that uniruled projective manifolds 
bearing a holomorphic $G$--structure modeled on a Hermitian symmetric spaces of rank $\,\geq\, 2$ are flat, and globally 
isomorphic to the corresponding Hermitian symmetric space endowed with its standard $G$--structure.

\subsection{Proof of Theorem \ref{classification}}\label{pt1}

Let $M$ be a compact K\"ahler--Einstein manifold of complex dimension $n\,>\,1$ endowed with a holomorphic projective 
connection. By the classification result of Kobayashi and Ochiai, \cite{KO,KO1}, $M$ is biholomorphic to one of the 
standard models: either to ${\mathbb C}{\mathbb P}^n$, or to a compact quotient of ${\rm \mathbb H}_{\mathbb C}^n$ by 
a torsion-free discrete subgroup of ${\rm S}{\rm U}(n,1)$ or to an \'etale quotient of a compact complex $n$-torus.

Corollary \ref{cor2} proves that the only holomorphic projective connection on ${\mathbb 
C}{\mathbb P}^n$ is the flat standard one (this was already known; see \cite{Ye,JR1,BM}).

Corollary \ref{cor1} proves that for $n\,>\,1$, the only holomorphic projective connection on compact quotients of 
${\rm \mathbb H}_{\mathbb C}^n$ by a torsion-free discrete subgroup of ${\rm S}{\rm U}(n,1)$ is the flat standard 
one. The result was already known for flat holomorphic projective connections \cite{MY} (see also \cite{Kl1});
Corollary \ref{cor1} uses arguments from their proof.

Proposition \ref{abelian} shows that a holomorphic projective connection on a complex compact $n$-torus
is represented by a global translation-invariant torsionfree holomorphic affine connection. Moreover, Proposition 
\ref{abelian} proves that, for $n\,>\,2$, the generic translation invariant holomorphic affine connection on a compact 
complex $n$-torus is not projectively flat. This completes the proof of Theorem \ref{classification}.

\section{Projective connections on Kuga-Shimura threefolds}\label{section 4}

As mentioned in the introduction, complex projective threefolds admitting holomorphic projective connections have 
been classified by Jahnke and Radloff in \cite{JR1}. Their result says that the only examples are either the 
standard ones or an \'etale quotient of a Kuga--Shimura projective threefold (see Section \ref{section 2}). The 
holomorphic projective connections on the standard examples were studied in Section \ref{section 3}.

In this Section we study holomorphic projective connections on Kuga--Shimura threefolds and prove Theorem 
\ref{main} and Corollary \ref{cor Lich}.

\begin{proposition}\label{parameters}
The projective equivalence classes of holomorphic projective connections on a Kuga--Shimura projective threefold 
$M\, \longrightarrow\, \Sigma$ are parametrized by a complex affine space on the vector space
${\rm H}^0(\Sigma,\, K^{\frac{3}{2}}_{\Sigma})^{\oplus 2}$, 
where $K_{\Sigma}$ is the canonical bundle of the base Riemann surface $\Sigma$ (see \eqref{ksf}).
\end{proposition}

\begin{proof}
Let
\begin{equation}\label{vp}
\varpi\, :\, \widetilde{M}\,=\,\mathcal H \times {\mathbb C}^2\, \longrightarrow\, M
\end{equation}
be the universal cover of $M$. Let $\phi$ be a holomorphic projective connection on $M$.
By Lemma \ref{trivial class}, the holomorphic projective connection $\varpi^*\phi$ 
on $\widetilde{M}$ is represented by a torsionfree holomorphic affine connection $\nabla$
on $\widetilde{M}$. 

Let $\nabla_0$ be the standard flat affine connection of $\mathcal H \times {\mathbb C}^2$ (seen as an
open subset in ${\mathbb C}^3$). Then
$$\Theta\,=\, \nabla - \nabla_0 \in {\rm H}^0(\widetilde{M},\, S^2(T^*\widetilde{M}) \otimes T\widetilde{M})\, ,$$
because both $\nabla$ and $\nabla_0$ are torsionfree.

We saw in the proof of Lemma \ref{lemma} that ${\mathbb F}(\Theta)\, \in\,
{\rm H}^0(\widetilde{M},\, (S^2(T^*\widetilde{M})\otimes T\widetilde{M})_0)^{\pi_1(M)}$, where
${\mathbb F}$ is the projection in \eqref{T}. Conversely, every element of
${\rm H}^0(\widetilde{M},\, (S^2(T^*\widetilde{M})\otimes T\widetilde{M})_0)^{\pi_1(M)}$
defines a holomorphic 
projective connection on $M$. Consequently, the space of projective equivalence classes of holomorphic projective 
connections on $M$ is identified with
${\rm H}^0(\widetilde{M},\, (S^2(T^*\widetilde{M})\otimes T\widetilde{M})_0)^{\pi_1(M)}$
(see the proof of Lemma \ref{lemma}).
We will first compute
$$
{\rm H}^0(\widetilde{M},\, S^2(T^*\widetilde{M})\otimes T\widetilde{M})^{\pi_1(M)}\, ,
$$
and then compute the locus in it of the trace-free ones.

As seen in Section \ref{ss Shimura}, $\pi_1(M)$ is a semi-direct product $\Gamma 
\ltimes \Lambda$, where $\Gamma$ is a Fuchsian group and $\Lambda \,\simeq\, {\mathbb Z}^{\oplus 4}$. We recall
from \eqref{s2} that the action of $\pi_1(M)$ on $\widetilde{M}$ is given by:
\begin{equation}\label{ss2}
(\gamma,\, \lambda) \cdot (\tau, z_1, z_2) \,\longmapsto\,
\left(\frac{a\tau +b}{c \tau +d},\, \frac{z_1+ m \tau +n}{c \tau +d},\, \frac{z_2 + k \tau +l}{c \tau +d}\right)
\end{equation}
for all $\gamma\,= \,\begin{pmatrix} a & b \\
c & d\end{pmatrix}\,\in\,\Gamma$ and $\lambda\,=\,(m,\,n,\,k,\,l)\,\in\, \Lambda$.
It follows that the action of the same element $ (\gamma,\, \lambda) \in \pi_1(M)$ on the standard
basis of holomorphic one-forms on $\widetilde{M}$ is given by:
$$(d \tau,\, dz_1,\, dz_2)
$$
$$
\,\longmapsto\, \left(\frac{1}{(c \tau +d)^2} d \tau,\, \frac{1}{c \tau +d}dz_1 -
\frac{cz_1-md + nc}{(c \tau +d)^2}d \tau, \, \frac{1}{c \tau +d}dz_2 - \frac{cz_2-kd+ lc}{(c \tau +d)^2}d
\tau\right)\, .$$
The action of the element $(\gamma,\, \lambda) \in \pi_1(M)$ on the dual basis is computed to be
the following:
$$(\frac{\partial}{\partial \tau},\, \frac{\partial}{\partial z_1},\, \frac{\partial}{\partial z_2})
\,\longmapsto \,\Big((c \tau +d)^2 \frac{\partial}{\partial\tau} + (c \tau +d)(cz_1 -md +nc)
\frac{\partial}{\partial z_1}
$$
$$
+\,(c \tau +d) (cz_2 -kd +lc) \frac{\partial}{\partial z_2},\, (c \tau +d) \frac{\partial}{\partial z_1},\,
(c \tau +d) \frac{\partial}{\partial z_2}\Big)\, .
$$

The expression for the general section of $S^2(T^*\widetilde{M}) \otimes T\widetilde{M}$ is a sum of type 
$$\sum_{i,j,k=1}^3f_{ij}^kdx_idx_j \frac{\partial}{\partial x_k}\, ,$$ where $x_1,\, x_2,\, x_3$ are chosen among the
global coordinates $(\tau,\, z_1,\, z_2) \,\in \,\mathcal H \otimes {\mathbb C}^2$ and $f_{ij}^k$ are holomorphic
functions defined on $\mathcal H \times {\mathbb C}^2$ such that $f_{ij}^k\,=\,f_{ji}^k$.
Therefore, in global coordinates $(\tau,\, z_1,\, z_2)$ of the universal cover $\widetilde{M}\,=\, \mathcal H \times 
{\mathbb C}^2$, the expression for the general section of $S^2(T^*\widetilde{M}) \otimes T\widetilde{M}$ is a sum 
of 18 terms:
$$f_{ \tau, \tau}^{z_1} d \tau \otimes d \tau \otimes \frac{\partial}{\partial z_1}+
f_{ \tau, \tau}^{z_2} d \tau \otimes d \tau \otimes \frac{\partial}{\partial z_2}
$$
$$
+ f_{ \tau, \tau}^{\tau} d \tau \otimes d \tau \otimes \frac{\partial}{\partial \tau} +
f_{z_1, \tau}^{z_1} dz_1 \otimes d \tau \otimes \frac{\partial}{\partial z_1}+
f_{z_2, \tau}^{z_2} d z_2 \otimes d \tau \otimes \frac{\partial}{\partial z_2}+\ldots $$
This tensor is trace-free if and only if the following three conditions hold:
$$f_{z_1,z_1}^{z_1} + f_{z_1, \tau}^{\tau } + f_{z_1,z_2}^{z_2} \,=\, 0\, ,$$
$$f_{z_2 ,z_1}^{z_1} + f_{z_2, \tau}^{\tau } + f_{z_2,z_2}^{z_2}\,=\, 0\, ,$$
$$f_{ \tau, z_1}^{z_1 } + f_{\tau , \tau }^{\tau}+ f_{\tau, z_2}^{z_2} \, =\,0\, .$$

The action of any $\lambda\,=\,(m,\,n,\,k,\, l) \,\in\, {\mathbb Z}^{\oplus 4}\,=\, \Lambda$ on
$\widetilde M$ is given by
$$
\lambda \cdot (\tau,\, z_1,\, z_2) \,\longmapsto\, (\tau,\, z_1+ m \tau +n,\, z_2+k\tau+l)
$$
(see \eqref{ss2}). In particular, the action of the normal subgroup $\Lambda\, \subset\, \pi_1(M)$
is trivial on the $\tau$-coordinate and it preserves the fibration defined by the projection
\begin{equation}\label{ef}
\mathcal H \times {\mathbb C}^2 \,\longrightarrow\, \mathcal H\, .
\end{equation}
The symbol functions $f_{ij}^k$ are evidently invariant under the action of
$\Lambda$. In particular, the functions $f_{ij}^k$ are constants on the fibers of the
projection in \eqref{ef}, i.e., every $f_{ij}^k$ depends only of the parameter $\tau$.

We now compute the components of the tensor $\Theta$ enforcing the invariance condition
under the action $\pi_1(M)$.

Take any $(\gamma,\, \lambda)\, \in\, \pi_1(M)$.
Identifying the coefficient of $d z_1 \otimes dz_2 \otimes \frac{\partial}{\partial z_1}$ in the
expressions of $\Theta$ and $(\gamma,\, \lambda)^*\Theta$ we get:
$$f_{z_1,z_2}^{z_1}(\tau) \,=\,\frac{1}{c \tau+d} f_{z_1,z_2}^{z_1} (\frac{a \tau +b}{c \tau +d})\, .$$
This implies that
$f_{z_1,z_2}^{z_1}$ is a holomorphic section of $K_{\Sigma}^{\frac{1}{2}}$, meaning the holomorphic weighted-form
$f_{z_1,z_2}^{z_1}(\tau) d \tau$ descends to $M$ and
the descended section coincides with the pull-back of a holomorphic section of 
$K_{\Sigma}^{\frac{1}{2}}$ through the Kuga--Shimura fibration in \eqref{ksf}.

Identifying the coefficient of $d \tau \otimes d \tau \otimes \frac{\partial}{\partial z_1}$
in the expressions of $\Theta$ and $(\gamma,\, \lambda)^*\Theta$ we get:
$$f_{\tau, \tau}^{z_1}(\tau) \,=\,\frac{1}{(c \tau+d)^3} f_{\tau , \tau }^{z_1} (\frac{a \tau +b}{c \tau +d}) + 
\frac{1}{(c \tau+d)^3} f_{z_1 , z_1 }^{z_1} (\frac{a \tau +b}{c \tau +d}) (cz_1-md +nc)^2+ 
$$
$$
\frac{1}{(c \tau+d)^3} f_{\tau , \tau }^{\tau} (\frac{a \tau +b}{c \tau +d}) (cz_1-md +nc)+
\frac{1}{(c \tau+d)^3} f_{z_1 , z_1 }^{\tau} (\frac{a \tau +b}{c \tau +d}) (cz_1-md +nc)^3
$$
$$
+ \frac{1}{(c \tau+d)^3} f_{z_2 , z_2 }^{z_1} (\frac{a \tau +b}{c \tau +d})
(cz_2 -kd +lc)^2
$$
$$
+2 \frac{1}{(c \tau+d)^3} f_{z_1 , z_2 }^{\tau} (\frac{a \tau +b}{c \tau +d})
(cz_1-md +nc)^2(cz_2-kd+lc)
$$
$$
+ 2 \frac{1}{(c \tau+d)^3} f_{z_1 , \tau }^{z_1} (\frac{a \tau +b}{c \tau +d})
(-cz_1 + md -nc) + 2 \frac{1}{(c \tau+d)^3} f_{z_2 , \tau }^{z_1} (\frac{a \tau +b}{c \tau +d}) (-cz_2 + kd -lc) 
$$
$$
- 2\frac{1}{(c \tau+d)^3} f_{z_1 , \tau }^{\tau } (\frac{a \tau +b}{c \tau +d}) (cz_1 -md +nc)^2-
2\frac{1}{(c \tau+d)^3} f_{z_2 , \tau }^{\tau} (\frac{a \tau +b}{c \tau +d}) (cz_1 - md +nc) (cz_2-kd+lc)
$$
$$
+2 \frac{1}{(c \tau+d)^3} f_{z_1 , z_2}^{z_1} (\frac{a \tau +b}{c \tau +d}) (cz_1 - md + nc) (cz_2-kd+lc)\, .
$$

Since the polynomial in the right hand side of the above equation must be independent  of $z_1$ and $z_2$, 
this yields
$$ f_{z_1 , z_1}^{\tau}\,=\, f_{z_1 , z_2}^{\tau}\,=\,f_{z_2 , z_1}^{\tau}\,=\,f_{z_2 , \tau}^{z_1}\,=\,
f_{ \tau, z_2}^{z_1}=f_{z_1 , \tau}^{z_2}\,=\,f_{\tau, z_1}^{z_2}\,=\,0\, ,$$
$$ f_{z_1 , z_2}^{z_1}\,=\,f_{z_2 , z_1}^{z_1}\,=\,f_{\tau , z_2}^{\tau}\,=\,f_{z_2, \tau}^{\tau}\, , $$
$$ f_{z_1 , z_1}^{z_1}\,=\,2 f_{z_1 , \tau}^{\tau}\,=\,2 f_{\tau , z_1}^{\tau}\, ,$$
$$f_{\tau, \tau}^{\tau}\,=\,2 f_{z_1, \tau}^{z_1}\,=\, 2 f_{\tau, z_1} ^{z_1}\, .$$
Also, we get that $f_{\tau, \tau}^{z_1}$ is a holomorphic section of $K^{\frac{3}{2}}_{\Sigma}$,
meaning the holomorphic weighted-form
$f_{\tau, \tau}^{z_1}$ descends to $M$ and
the descended section coincides with the pull-back of a holomorphic section of
$K_{\Sigma}^{\frac{3}{2}}$ through the Kuga--Shimura fibration in \eqref{ksf}.

Performing the same computation for the coefficient of $d \tau \otimes d \tau \otimes \frac{\partial}{\partial 
z_2}$ we get that
$$f_{z_2 , z_2}^{\tau}\,=\,0\, ,$$
$$f_{z_1 , z_2}^{z_2}\,=\,f_{z_2 , z_1}^{z_2}\,=\,f_{\tau , z_1}^{\tau}\,=\,f_{z_1, \tau}^{\tau}\, ,$$
$$f_{z_2 , z_2}^{z_2}\, =\,2 f_{z_2 , \tau}^{\tau}\,=\,2 f_{\tau , z_2}^{\tau}\, ,$$
$$f_{\tau, \tau}^{\tau}\,=\,2 f_{z_2, \tau}^{z_2}\,=\, 2 f_{\tau, z_2} ^{z_2}\,.$$
We also get that $f_{\tau, \tau}^{z_2}$ is a holomorphic section of $K^{\frac{3}{2}}_{\Sigma}$ in the
sense explained above.

We now identify the coefficient of $d \tau \otimes d\tau \otimes \frac{\partial}{\partial \tau}$. Since we already 
know that $f_{z_1, z_1}^{\tau}\,=\,f_{z_2, z_2}^{\tau}\,=\,f_{z_1, z_2}^{\tau}\,=\,0$, the equation is:
$$f_{\tau, \tau}^{\tau}(\tau) \,=\, f_{\tau, \tau}^{\tau}(\frac{a \tau +b}{c\tau +d}) \frac{1}{(c\tau +d)^2} +
2 f_{z_1, \tau}^{\tau} (\frac{a \tau +b}{c \tau +d}) \frac{1}{(c \tau +d)^2}(-cz_1 +md-nc)
$$
$$
+ 2f_{z_2, \tau}^{\tau} (\frac{a \tau +b}{c \tau +d}) \frac{1}{(c \tau +d)^2}(-cz_2 +kd-lc)\, .$$
This implies that $f_{z_1, \tau}^{\tau}\,=\, f_{z_2, \tau}^{\tau}\,=\,0$, and $f_{\tau, \tau}^{\tau}$
is a holomorphic section of $K_{\Sigma}$ in the sense explained above.

Let us now identify the coefficient of $d z_1 \otimes dz_1 \otimes \frac{\partial}{\partial z_1}$. Since $f_{z_1, 
z_2}^{\tau}\,=\,0$, we get that
$$f_{z_1, z_1}^{z_1}(\tau) = f_{z_1, z_1}^{z_1}(\frac{a \tau +b}{c\tau +d}) \frac{1}{(c\tau 
+d)}\, .$$ It follows that $f_{z_1, z_1}^{z_1}$ is a holomorphic section of $K_{\Sigma}^{\frac{1}{2}}$
in the sense explained above.

Consider the coefficient of $d z_2 \otimes dz_2 \otimes \frac{\partial}{\partial z_2}$; we obtain the same 
result for $f_{z_2, z_2}^{z_2}$.

The conclusion is that the only non-vanishing coefficients are:
$$f_{\tau , \tau }^{\tau}\,=\,2f_{z_1, \tau }^{z_1} \,=\,2f_{z_2, \tau }^{z_2} \,\in\,
{\rm H}^0(\Sigma,\, K_{\Sigma})\, ;$$
$$f_{\tau , \tau }^{z_1},\,\ f_{\tau , \tau }^{z_2} \,\in\, {\rm H}^0(\Sigma,\,K^{\frac{3}{2}}_{\Sigma})\, .$$
Therefore, we have a canonical identification
\begin{equation}\label{ca}
{\rm H}^0(\Sigma,\,K^{\frac{3}{2}}_{\Sigma})^{\oplus 2}\oplus {\rm H}^0(\Sigma,\,K_{\Sigma})
\, \stackrel{\sim}{\longrightarrow}\,
{\rm H}^0(\widetilde{M},\, S^2(T^*\widetilde{M})\otimes T\widetilde{M})^{\pi_1(M)}\, .
\end{equation}
For $A,\, B\, \in\, {\rm H}^0(\Sigma,\,K^{\frac{3}{2}}_{\Sigma})$ and
$C\, \in\, {\rm H}^0(\Sigma,\,K_{\Sigma})$, the corresponding element
\begin{equation}\label{ca2}
\Theta_{A,B,C}\, \in\,
{\rm H}^0(\widetilde{M},\, S^2(T^*\widetilde{M})\otimes T\widetilde{M})^{\pi_1(M)}
\end{equation}
by the isomorphism in \eqref{ca} is given by
$$f_{\tau , \tau }^{z_1}\,=\, A\, ,\ \ f_{\tau , \tau }^{z_2}\,=\, B,\ \
f_{\tau , \tau }^{\tau}\,=\,2f_{z_1, \tau}^{z_1 }\,=\, 2f_{z_2, \tau }^{z_2} \,=\,C\, .$$

It can be shown that $\nabla^{A,B,C}$ and $\nabla^{A,B,0}$ are
projectively equivalent. Indeed, define 
$\phi_{\tau}$ to be the holomorphic one-form on $\mathcal H \times {\mathbb C}^2$ such that
$$\phi_{\tau} 
(\frac{\partial}{\partial \tau})\,=\,\frac{1}{2}C\ \ \text{ and }\ \
\phi_{\tau} (\frac{\partial}{\partial z_i})\,=\, 0$$
for $i=1,2$. It follows that
\begin{equation}\label{ti}
\nabla^{A,B,C}_XY -\nabla^{A,B,0}_XY\,=\, \phi_{\tau}(X)(Y)+ \phi_{\tau}(Y)X
\end{equation}
for all locally defined holomorphic 
vector fields $X,\, Y$; see \eqref{ai}. From \eqref{ti} it follows that $\nabla^{A,B,C}$ and
$\nabla^{A,B,0}$ are projectively equivalent.

For the section $\Theta_{A,B,C}$ in \eqref{ca2},
we have ${\rm div}(\Theta_{A,B,C})\,=\,2C d\tau$, where ${\rm div}$ is constructed as
in \eqref{x3}. Consequently, the trace-free condition ${\rm 
div}(\Theta_{A,B,C})\,=\,0$ holds if and only if $$C\,=\,0\, .$$
In fact, the trace-free part of $\Theta^{A,B,C}$ is 
$${\mathbb F} (\Theta_{A,B,C})\,=\, \Theta_{A,B,0} \,\in\,
H^0(\widetilde{M},\, (S^2(T^*\widetilde{M}) \otimes T^*\widetilde{M})_0)\, ,$$
where $\mathbb F$ is the projection in \eqref{T}.

Therefore, from the isomorphism in \eqref{ca} we have a canonical identification
\begin{equation}\label{ca3}
{\rm H}^0(\Sigma,\,K^{\frac{3}{2}}_{\Sigma})^{\oplus 2} \, \stackrel{\sim}{\longrightarrow}\,
{\rm H}^0(\widetilde{M},\, (S^2(T^*\widetilde{M})\otimes T\widetilde{M})_0)^{\pi_1(M)}
\end{equation}
that sends any $(A,\, B)\,\in\, {\rm H}^0(K^{\frac{3}{2}}_{\Sigma})^{\oplus 2}$
to $\Theta_{A,B,0}$ in \eqref{ca2}.

Recall that $\nabla^{A,B,0}\,=\, \nabla_0 + \Theta_{A,B,0}$ is a $\pi_1(M)$--invariant
holomorphic projective connection on 
$\widetilde{M}$. Therefore, it descends to a holomorphic projective connection $\nabla^{A,B}$ on $M$. Moreover, we 
have seen that the space of projective equivalence classes of holomorphic projective connections on $M$
is identified with ${\rm H}^0(\widetilde{M},\, (S^2(T^*\widetilde{M}) \otimes T^*\widetilde{M})_0)^{\pi_1(M)}$.
Therefore the projective connection $\nabla^{A,B}$ is projectively equivalent to $\nabla^{A',B'}$ if and 
only if $(A,\,B)\,=\,(A',\,B')$; we take the base projective connection
$\nabla^{0,0}$ to be the standard flat projective connection
induced by the open embedding of $\widetilde{M}$ in ${\mathbb C}{\mathbb P}^3$.
\end{proof}

\begin{remark}\label{KS} There are Kuga--Shimura projective threefolds over Shimura (compact) curves of 
arbitrarily large genus (see, for instance, \cite[Section 5]{KV}). Using Riemann-Roch theorem we deduce that 
$\dim {\rm H}^0(\Sigma,\, K^{\frac{3}{2}}_{\Sigma})^{\oplus 2}\,=\, 4(\text{genus}(\Sigma)-1)$, so it
can be arbitrarily large. This implies that the space of projective classes of 
holomorphic projective connections on a Kuga--Shimura projective threefold can have dimension arbitrarily large.
The next Proposition \ref{curvature} shows that all these holomorphic projective connections are flat. 
\end{remark}

\begin{proposition} \label{curvature}
Let $M$ be a projective Kuga--Shimura threefold which fibers over a Shimura curve $\Sigma$. Then all holomorphic 
projective connections $\phi$ on $M$ are flat. The fibers of the Kuga--Shimura fibration are (flat) totally 
geodesic for $\phi$.
\end{proposition}

\begin{proof}
Take a holomorphic projective connection $\phi$ on $M$. Let
$$(A,\, B) \,\in\,
{\rm H}^0(K^{\frac{3}{2}}_{\Sigma})^{\oplus 2}$$ be the pair associated to $\phi$ by
Proposition \ref{parameters}.
Recall that $\varpi^* \phi$ (see \eqref{vp}) is represented by a 
holomorphic affine connection $$\nabla^{A,B}\,:=\, \nabla^{A,B,0}\,=\, \nabla_0 + \Theta_{A,B}\, ,$$
where $\Theta_{A,B}\, :=\, \Theta_{A,B,0}\, \in\,H^0(\widetilde{M},\, (S^2(T^*\widetilde{M})
\otimes T^*\widetilde{M})_0)^{\pi_1(M)}$
corresponds to the Christoffel symbols
$$(f_{\tau , \tau }^{z_1},\,f_{\tau , \tau }^{z_2}) \,=\,(A,\, B)\,\in\,
{\rm H}^0(\Sigma,\, K^{\frac{3}{2}}_{\Sigma})^{\oplus 2}$$
(see \eqref{ca2}).
As noted in the proof of Proposition \ref{parameters},
the holomorphic weighted-forms $f_{\tau , \tau }^{z_1}(\tau) d\tau$ (respectively, $f_{\tau , \tau }^{z_2}(\tau) 
d\tau$) defined on $\widetilde{M}\,=\,\mathcal{H} \times {\mathbb C}^2$ descends to $M$ as the
holomorphic section $A$ (respectively, $B$) of $K^{\frac{3}{2}}_{\Sigma}$.

We now compute the curvature of the associated affine connection $\nabla^{A,B}\,=\, \nabla_0 + \Theta_{A,B}$.

This computation is formally the same as the computation of the curvature of the connection 
$\nabla^{A,B,0,0,0}$ in the proof of Proposition \ref{abelian}: see the proof of Lemma \ref{lemma appendix} in Appendix. The only difference is that here $A$ and $B$ 
depend on the variable $\tau$, while in the proof of  Lemma \ref{lemma appendix}   the parameters $A$ and $B$ are two 
constants. Nevertheless, the curvature tensor $R(X,\,Y)Z$ being anti-symmetric in variables $(X,Y)$ we have
$R(\frac{\partial}{\partial \tau}, \, \frac{\partial}{\partial \tau}) \,=\,0$; note that
$R(\frac{\partial}{\partial \tau}, \,\frac{\partial}{\partial \tau})$ is an endomorphism
of $T^*\widetilde{M}$. Therefore, the components of the 
affine curvature tensor $R$ do not depend on the derivatives of the functions
$\tau \,\longmapsto\, A(\tau)$ and $\tau \,\longmapsto\,B(\tau)$.

It follows that the computation of the curvature tensor $R$ is the same as for $\nabla^{A,B,0,0,0}$ in the proof of Lemma \ref{lemma appendix} in Appendix. In the case where $C\,=\,D\,=\,E\,=\,0$, the conclusion of the computations in the proof of 
Lemma \ref{lemma appendix}  is that the tensor $R$, and hence $W$, vanishes identically. This proves that 
$\nabla^{A,B}$ is projectively flat.

Moreover, $\nabla^{A,B}$ preserves the holomorphic two-dimensional foliation $\mathcal F$ of $\widetilde{M} 
\,=\,\mathcal{H} \times {\mathbb C}^2$ defined, in global coordinates $(\tau,\,z_1,\,z_2)$, by $d \tau\,=\,0$.
The connection $\nabla^{A,B}$ coincide with $\nabla_0$ when restriction to $\mathcal F$. More precisely, each leaf of
$\mathcal F$ is a fiber $\{h 
\} \times {\mathbb C}^2 \,\subset\, \mathcal H \times {\mathbb C}^2$, and the restrictions of
$\nabla_{A,B}$ and $\nabla_0$ to each fiber coincide. Observe that the projection on $M$ of the $\mathcal 
F$--leafs are exactly the fibers of the Kuga--Shimura fibration in \eqref{ksf}.
\end{proof}

\subsection{Proofs of Theorem \ref{main} and Corollary \ref{cor Lich}}\label{ptc}\mbox{}

\begin{proof}[{Proof of Theorem \ref{main}}] Let $M$ be a Kuga--Shimura projective threefold. Proposition 
\ref{parameters} provides a proof of statement (i) of the theorem, showing that the space of the projective classes of 
holomorphic projective connections on $M$ is identified with a complex affine space for the
${\rm H}^0(\Sigma,\, K^{\frac{3}{2}}_{\Sigma})^{\oplus 2}$.

Proposition \ref{curvature} proves statement (ii).
\end{proof}

Now we can deduce Corollary \ref{cor Lich}.

\begin{proof}[{Proof of Corollary \ref{cor Lich}}] Let $M$ be a projective threefold endowed with a holomorphic 
projective connection $\phi$. The main result in \cite{JR1} proves that $M$ is either a K\"ahler--Einstein threefold 
(and hence one of the standard examples \cite{KO,KO1}), or it is an \'etale quotient of a Kuga--Shimura projective 
threefold. Theorem \ref{main} shows that on Kuga--Shimura projective threefolds (and hence also on their \'etale 
quotients), all holomorphic projective connections are flat. Theorem \ref{classification} implies that on 
K\"ahler--Einstein projective threefolds $\phi$ is flat, except for abelian varieties (and their \'etale quotients) 
on which $\phi$ is translation invariant (but generically not flat).
\end{proof} 
 
\section{Simply connected complex threefolds with trivial canonical bundle} \label{section 5}

In this final section we deal with compact complex manifolds of dimension three. For any cocompact lattice 
$\Gamma\, \subset\, {\rm SL}(2, {\mathbb C})$, the quotient ${\rm SL}(2, \mathbb C) / \Gamma$
is compact non-K\"ahler threefold with trivial tangent bundle admitting a flat
holomorphic projective connection. Ghys constructed their deformations that have canonical bundle trivial
and admit a flat holomorphic projective connection \cite{Gh} (see also \cite[Section 5]{BD}).

A simply connected manifold $M$ with trivial canonical bundle does not admit a flat holomorphic projective 
connection. Indeed, the developing map of such a holomorphic projective structure would realize a
biholomorphism between $M$ and the complex projective space (which has nontrivial first Chern class).
For dimension three, the following stronger result holds.

\begin{theorem}\label{thm sc}
A simply connected compact complex threefold with trivial canonical bundle does not admit any holomorphic
projective connection.
\end{theorem}

To prove Theorem \ref{thm sc},
we will make use of the theory of {\it rigid geometric structure} as developed in \cite{Gr,DG}. A holomorphic  torsionfree
affine connection is known to be rigid (of order one) in the sense of \cite{Gr,DG} because a local automorphism of 
the connection is completely determined by its underlying one-jet at any given point (i.e., the differential sends 
parametrized holomorphic geodesic curves to parametrized holomorphic geodesic curves). A holomorphic projective 
connection is known to be rigid (of order two) in the sense of \cite{Gr,DG} because a local automorphism of the 
connection is completely determined by its underlying two-jet at any given point.

The algebraic dimension $a(N)$ of a compact complex manifold $N$ is the transcendence degree of the field of 
meromorphic functions ${\mathbb C}(N)$ over $\mathbb C$ (see \cite[p.~24, Chapter 3]{Ue}). We have
$a(N)\, \leq\, \dim N$, and $a(N)\, =\, \dim N$ if and only if $N$ 
is bimeromorphic to a complex projective variety \cite{Mo}. The manifolds with maximal algebraic dimension 
are called {\it Moishezon manifolds}.

The Killing Lie algebra of a holomorphic rigid geometric structure (here the geometric structure is a holomorphic 
projective or affine connection) on a compact complex manifold $N$ of complex dimension $n$ has generic orbits of complex 
dimension $\geq\, n-a(N)$ \cite[p.~568, Theorem 3]{Du}.
This result will be useful in the proof of the following propositions.

\begin{proposition}\label{Killing}
Any holomorphic projective connection $\phi$ on a compact complex
threefold $M$ admits a nontrivial Killing Lie algebra.
\end{proposition}

\begin{proof}
Assume, by contradiction, that the Killing Lie algebra
for $\phi$ is trivial. Now Theorem 3 in \cite{Du} says that the manifold $M$ is Moishezon. But 
Moishezon manifolds
bearing a holomorphic Cartan geometry (in 
particular, a holomorphic affine connection \cite{Sh}) are known to be projective (see Corollary 
2 in \cite{BM}).

On the other hand, Corollary \ref{cor Lich} implies that the Killing Lie algebra of a 
holomorphic projective connection on a projective threefold is transitive, hence it has 
dimension at least three: a contradiction.
\end{proof}

\begin{proposition} \label{Killing2} Let $M$ be a compact complex manifold with trivial canonical bundle bearing a 
holomorphic projective connection $\phi$. Then the following five hold:
\begin{enumerate}
\item[(i)] The Killing Lie algebra of $\phi$ is nontrivial.\\

In the following four, assume that $M$ is simply connected.\\

\item[(ii)]
The automorphism group $\rm{Aut} (M,\, \phi)$ of $(M,\, \phi)$ is a complex Lie group of positive dimension.

\item[(iii)] A maximal connected abelian complex subgroup $A$ in
$\rm{Aut} (M, \,\phi)$ has positive dimension.

\item[(iv)] The $A$--orbits in $M$ coincide with those of the maximal (real) compact subgroup $K(A)\,\subset\, A$; they 
all are compact complex tori.

\item[(v)] The $A$--action on $M$ does not admit any fixed point.
\end{enumerate}
\end{proposition} 

\begin{proof}
Denote by $n$ the complex dimension of $M$. Consider a holomorphic volume form $\omega$ on $M$; this
means that $\omega$ is 
a holomorphic trivializing section of the canonical bundle $K_M$. Lemma \ref{trivial class}, and its proof, 
imply that there is a unique torsionfree affine connection $\nabla$ on $M$ such that
\begin{itemize}
\item $\nabla$ is projectively equivalent to $\phi$, and

\item $\omega$ is parallel with respect to $\nabla$.
\end{itemize}
Let us first prove that that $M$ does not admit any nontrivial rational curve.

Take a holomorphic map $f\,: \,{\mathbb C}{\mathbb P}^1\,
\longrightarrow\, M$ and consider the pull-back $f^*TM$ equipped with the holomorphic
connection $f^*\nabla$.
The connection $f^*\nabla$ is flat (because its curvature is a (bundle valued) holomorphic
two-form). This implies that $f^*TM$ is holomorphically trivial, because ${\mathbb C}
{\mathbb P}^1$ is simply connected. Since $\text{degree}(T{\mathbb C}{\mathbb P}^1)\, >\, 0$,
there is no nonzero homomorphism from
$T{\mathbb C}{\mathbb P}^1$ to the trivial bundle $f^*TM$. In particular, the differential $df$ of $f$
vanishes identically. This implies that $f$ is a constant map.

(i) To prove by contradiction, assume that the Killing Lie algebra of $\phi$ is trivial (in particular the Killing 
Lie algebra of $\nabla$ is also trivial). Then Theorem 3 in \cite{Du} implies that $M$ is a Moishezon manifold. But 
Moishezon manifolds with no rational curves are known to be projective \cite[p.~307, Theorem 3.1]{Cas}. A complex 
projective manifold bearing a holomorphic affine connection $\nabla$ is covered by an abelian variety \cite{IKO}. 
Moreover, the pull-back of $\nabla$ to the covering abelian variety is translation invariant (see Proposition 
\ref{abelian}). In particular, $\nabla$ is locally homogeneous, and hence $\phi$ is also locally homogeneous; so the 
Killing Lie algebra of $\phi$ is transitive on $M$: a contradiction.

Now assume that $M$ is simply connected.

(ii) Each element in the Killing Lie algebra for $\phi$ extends to a global holomorphic Killing vector field for 
$\phi$ (defined on entire $M$): this was first proved by Nomizu in \cite{No} for Killing vector fields of analytic 
Riemannian metrics, and subsequently it was generalized to $G$--structures \cite{Am}, to rigid geometric structures 
\cite{DG,Gr} and also to Cartan geometries \cite{Me,Pe}. This implies that there is a nontrivial connected complex 
Lie group $G$ acting by biholomorphisms on $M$ that preserves the holomorphic projective connection $\phi$. This 
group $G$ coincides with the connected component, containing the identity element, of the automorphism group 
$\rm{Aut}(M,\, \phi)$. The Lie algebra ${\rm Lie}(G)$ of $G$ is the Lie algebra of global holomorphic Killing vector 
fields for $\phi$.

iii) The nonzero holomorphic section $\omega$ of the canonical bundle $K_M$ defines a smooth 
real volume form on $M$ given by $(\sqrt{-1})^n \cdot \omega \wedge \overline{\omega}$.

We will prove that the action of the group $G$ preserves the smooth measure
$(\sqrt{-1})^n \cdot \omega \wedge \overline{\omega}$ on $M$. To prove this, consider a 
holomorphic Killing vector field $X\, \in\, {\rm Lie}(G)$. The Lie derivative $L_{X} \omega$ of $\omega$ is a 
holomorphic section of $K_{M}$. So there is a constant $c \,\in\, \mathbb C$ such that $L_{X} \omega\,=\, 
c \cdot \omega$. Hence, if $\Psi^t$ is the one-parameter subgroup of $G$ generated by $X$, 
we get that $$(\Psi^t)^{*} \omega \,=\, \exp(ct) \cdot \omega$$ for all $t \,\in\, \mathbb C$; 
recall that any holomorphic vector field on a compact manifold is complete, and therefore its 
flow is defined on all of $\mathbb C$. Since the total volume $\int_M (\sqrt{-1})^n \cdot 
\omega \wedge \overline{\omega}$ of the manifold $M$ is invariant by any automorphism, it 
follows that $|\exp(ct)|\,=\, 1$ for all $t \,\in\, \mathbb C$. By Liouville 
Theorem, the entire function $t \, \longmapsto \exp(ct)$ must be constant and equal to $1$ 
(the value of the function at $t\,=\,0$ being $1$). This implies that $c\,=\,0$, and $\omega$ is 
$X$-invariant. Since the complex Lie group $G$ is connected, it is generated by the flows of 
its fundamental vector fields. It follows that every element of $G$ preserves the volume form.

Moreover since $G$ preserves the holomorphic volume form $\omega$ and the holomorphic  projective connection $\phi$, the action 
of $G$ also preserves the associated torsionfree holomorphic affine connection $\nabla$ representing $\phi$; it was 
observed in the proof of Lemma \ref{trivial class} that $\nabla$ is canonically associated to $\phi$ and $\omega$.

Let us now apply the Gromov abelianization trick (see \cite[Section 3.2.A]{DG} or \cite{Gr}) and consider the 
rigid geometric structures which is a juxtaposition of the holomorphic projective connection 
$\phi$ with a family of global holomorphic vector fields $\{X_1,\, \cdots,\, X_{k}\}\,\in\, 
{\rm H}^0(M,\, TM)$ forming a basis of the Lie algebra of $G$, seen as a subalgebra of $TM$. 

Denote by $A$ the connected component of the identity element in the automorphism group of the holomorphic rigid 
geometric structure $$\phi' \,=\,(\phi,\, X_1,\, \cdots,\, X_k)\, .$$ Then $A$ is a maximal connected abelian 
complex Lie subgroup in $\rm{Aut}(M, \,\phi)$ (see \cite[Section 3.1 Lemma]{DA}, for more details).

Applying \cite[Theorem 3]{Du} to $\phi'$ it is deduced that $A$ acts with generic orbits of complex 
dimension at least $n-a(M)$.  As above, $M$ is not Moishezon, so $n-a(M)\,>\,0$.  Indeed, recall that we have seen in the proof of point (i) that $M$ Moishezon implies $M$ is covered by an abelian variety: a contradiction (since $M$ is simply connected).

(iv) Since $A$ preserves a smooth measure on $M$, its orbits are compact and coincide with the orbits of its maximal 
(real) compact subgroup $K(A)\, \subset\, A$ (see \cite[Section 3.7]{Gr} and \cite[Section 3.5.4]{DG}).

Choose a point $m_0 \,\in\, M$, and consider its $A$--orbit $Am_0$. Then $Am_0$ is biholomorphic to the homogeneous 
space $A/A_{m_0}$, where $A_{m_0}$ is the complex subgroup of $A$ that fixes $m_0$. Any basis of 
the quotient space ${\rm Lie}(A)/{\rm Lie} (A_{m_0})$ is invariant by the adjoint representation of 
$A_{m_0}$ (because $A$ is abelian) and provides a holomorphic trivialization of the holomorphic tangent bundle of 
the homogeneous space $A/A_{m_0}$. So $A/A_{m_0}$ is a compact parallelizable manifold \cite{Wa}. Moreover, the 
holomorphic tangent bundle is trivialized by commuting vector fields. Therefore $A/A_{m_0}$ is a compact complex 
torus. Consequently, all $A$--orbits are compact complex tori (notice that some orbits could be of dimension zero: 
these are fixed points of the $A$--action).

(v) To prove this by contradiction, assume that $m_0 \,\in \, M$ is fixed by the action of $A$ on $M$. To any $g\, 
\in \,A$ associate its differential $dg(m_0)\,\in\, {\rm GL}(T_{m_0}M)$; this gives the isotropy homomorphism $i\, 
:\,A \, \longrightarrow\, {\rm GL}(T_{m_0}M)$ at $m_0$. Moreover, since $A$ preserves the holomorphic torsionfree 
affine connection $\nabla$, the $A$ action on $M$ is linearizable in local holomorphic $\nabla$--exponential 
coordinates in the neighborhood of $m_0$. More precisely, there exists an open neighborhood $U$ of $0\,\in\, 
T_{m_0}M$ and an open neighborhood $U'$ of $m_0\,\in\, M$ and a biholomorphism $\beta\, :\, U\, \longrightarrow\, 
U'$, such that $\beta$ intertwines the actions of $i(A)$ on $U$ and of $A$ on $U'$. In particular, the homomorphism 
$i$ is injective (i.e., the isotropy representation of $A$ is faithful).

It is known that $i(A)\,\subset\, {\rm GL}(T_{m_0}M)$ is a complex algebraic subgroup; this
is because $i(A)$ coincides with the 
stabilizer of a $k$-jet of the rigid geometric structure $\phi'$ (for $k \in \mathbb{N} $ large enough) and the 
${\rm GL}(T_{m_0}M)$--action on the space of $k$-jets of $\phi'$ at $m_0$ is algebraic (see \cite[Sections 3.5 and 
3.7]{Gr} and \cite[Sections 3.2A and 3.5]{DG} or \cite[Theorem 3.11]{Me}).

As before , $K(A)\, \subset\, A$ is the maximal compact subgroup.
Let $K(A)^0$ be the connected component of $K(A)$ containing the identity element.
It is isomorphic to $(S^1)^\ell\,=\, {\rm U}(1)^\ell$ for some $\ell\, \geq\, 1$.
Let $i(K(A))_{\mathbb C}\, \subset\, i(A)$ be the complex Zariski closure of $i(K(A)^0)$.
This group $i(K(A))_{\mathbb C}$ is isomorphic to
$({\mathbb C}^*)^\ell$. We deduce that $T_{m_0}M$ splits as 
a direct sum
$$ T_{m_0}M\,=\,L_1 \oplus L_2 \oplus \ldots \oplus L_n$$
of complex $i(K(A))_{\mathbb C}$--invariant lines, such 
that $i(K(A))_{\mathbb C}$ acts on each $L_j$, $1 \,\leq\, j \,\leq \,n$, through a 
multiplicative character $$
\chi_j\, :\, ({\mathbb C}^*)^\ell\, \longrightarrow\, {\mathbb C}^*\, ,\ \
(t_1,\, \cdots,\, t_\ell) \,\longmapsto\, t_1^{n^j_1} \cdot t_2 ^{n^j_2} 
\cdot \ldots \cdot t^{n^j_\ell}_\ell$$
defined by a given $(n^j_1,\, \cdots,\, n^j_\ell)\,\in\, {\mathbb Z}^\ell$.

Since the isotropy representation $i$ is faithful, it follows that at least one of the characters $\{\chi_1,\, 
\chi_2,\, \ldots \, , \chi_n\}$ is nontrivial. Recall that $\beta$ intertwines the actions of $i(K(A))_{\mathbb C}$ 
on $U$ and $A$ on $U'$. Then there exists points in $U' \setminus \{m_0 \}$ 
such that the $A$--orbit of any of them accumulates at the fixed point $m_0$. This contradicts the fact that
the $A$--orbits in 
$M$ are compact. Therefore, the $A$--action on $M$ does not have any fixed points.
\end{proof}

Let us now give a proof of Theorem \ref{thm sc}.

\begin{proof}[{Proof of Theorem \ref{thm sc}}] To prove by contradiction, assume that there is a complex compact 
threefold $M$ with trivial canonical bundle $K_M$ and bearing a holomorphic projective connection $\phi$. We are 
exactly in the situation described by Proposition \ref{Killing2}. We keep the same notations as in Proposition 
\ref{Killing2}. In particular, we have a holomorphic torsionfree affine
connection $\nabla$ on $M$ representing $\phi$, which is preserved by the 
action of the nontrivial connected abelian group $A$ of automorphisms.

It was proved in Proposition \ref{Killing2}(iv) that all $A$--orbits are biholomorphic to compact complex 
tori. Since $M$ is not homeomorphic to a compact complex torus (it is simply connected), we deduce that the generic 
$A$--orbits are either of complex dimension two, or of complex dimension one.

These two cases will be dealt separately.

{\it Case of generic $A$-orbits being of complex dimension two.}\, Take 
holomorphic Killing vector fields $(X_1,\, X_2)$ which span the tangent space to the $A$--orbit at a 
generic point.
As before, let $\omega$ be a nonzero holomorphic section of $K_M$. We have the holomorphic one--form
$\theta$ on $M$ defined by
$$
\theta(x)(v)\,=\, \omega(x)(v, X_1(x), X_2(x))
$$
for all $x\,\in\, M$ and $v\, \in\, T_xM$. This one--form $\theta$ is $A$--invariant and it vanishes on the 
$A$--orbits. Since the kernel of $\theta$ coincides, at the generic point, with the holomorphic tangent space of
the foliation defined by the $A$--action, the one--form $\theta$ satisfies the Frobenius integrability condition 
$\theta \wedge d \theta \,=\,0$. Moreover, it can be proved that this nontrivial one-form $\theta$ is closed.
See the proof of Theorem 4.4 in \cite{BD2} where it is shown that $\theta$ is closed because it is projectable
on a compact 
curve; this can also be deduced from the description of non-closed integrable one-forms on threefolds given in 
\cite[Proposition 3]{Br}. This implies that $H^1(M,\, {\mathbb C})\, \not=\, 0$, and hence the abelianization of 
the fundamental group of $M$ is infinite: a contradiction.

{\it Case of generic $A$-orbits being of complex dimension one.} Proposition \ref{Killing2}(v) proves
that the $A$--action on $M$ does not have any fixed points. It follows from Proposition \ref{Killing2}(iv) that
all $A$--orbits are elliptic curves, on which $K(A)$ acts transitively. 

We will prove that $K(A)$ acts freely on $M$.

To prove this, take any $m_0 \,\in\, M$, and let $I(m_0)\, \subset\, K(A)$ be the stabilizer of $m_0$ for the 
action of $K(A)$ on $M$. Then $I(m_0)$ is a compact abelian group fixing $m_0$. Its action linearizes in local 
holomorphic coordinates at $m_0$. For any $k \,\in\, I(m_0)$, since $K$ is abelian, the differential $dk(m_0)$ acts 
trivially on $T_{m_0} (K(A)(m_{0})) \,=\,T_{m_0}(Am_0)$; recall that any $A$--orbit is also a $K(A)$--orbit. On the 
other hand, the differential $dk(m_0)$ acts trivially on
the quotient space $T_{m_0}M /(T_{m_0}(Am_0))$, because any element of $A$ (in 
particular $k$) fixes (globally) each $A$--orbit (i.e., it acts trivially on the space of $A$--orbits). Since compact 
groups are reductive, it follows that the differential of $k$ is trivial. The isotropy representation at $m_0$ 
being faithful (see the proof of Proposition \ref{Killing2}(v)), this implies that $k$ is the identity element. 
Consequently, $I(m_0)$ is trivial, and the $K(A)$--action on $M$ is free.

It follows that $M$ is the total space of a
real principal $K(A)$--bundle over a smooth real manifold $B\,=\,M/K(A)$. The $K(A)$--orbits are complex manifolds,
because they are $A$--orbits. This implies that $B$ is also a complex 
manifold and the projection
\begin{equation}\label{de}
\delta \,:\, M \,\longrightarrow\, B\,=\,M/K(A)
\end{equation}
is a holomorphic submersion whose fibers are elliptic curves.

We will now prove that the fibration $\delta$ in \eqref{de} is a holomorphic principal elliptic
curve bundle over $B$.

The space of elliptic curves $C$ together with a symplectic basis of ${\rm H}^1(C,\, {\mathbb Z})$
is parametrized by Poincar\'e upper half-plane $\mathcal H$. The base $B$ in \eqref{de} is
simply connected, because $M$ is so (and the fibers of $\delta$ are connected). Therefore, fixing a point $b_0\, \in\, B$ and a symplectic
basis of ${\rm H}^1(\delta^{-1}(b_0),\, {\mathbb Z})$,
we get a holomorphic map
$$
\Phi\, :\, B\, \longrightarrow\, {\mathcal H}
$$
for the family of elliptic curves in \eqref{de}.
Since $B$ is compact, this $\Phi$ is a constant function.

Therefore, the fibration $\delta$ in \eqref{de} is isotrivial. By the fundamental result of Fischer and Grauert 
$\delta$ is a holomorphic bundle. Moreover since the fibers of $\delta$ are isomorphic elliptic curves on which 
$K(A)$ acts freely and transitively (by biholomorphisms), for any point $m_0 \,\in\, M$, the orbital map $K(A) 
\,\longrightarrow\, K(A)m_0$ induces on $K(A)$ the same complex structure (that of the fiber type of $\delta$). 
Hence $K(A)$ gets the complex structure of an elliptic curve for which the $K(A)$--action on $M$ is 
holomorphic; this elliptic curve will be denoted by ${\mathcal K}(A)$.
Therefore $\delta$ is a holomorphic principal ${\mathcal K}(A)$--bundle.

Since $\delta$ is a holomorphic principal ${\mathcal K}(A)$--bundle, and $K_M$ is holomorphically trivial,
it follows that $K_B$ is holomorphically trivial. It was noted above that $B$ is simply connected.
So $B$ is a K3 surface.

Theorem 1.1 of \cite{BD3} implies that the holomorphic affine connection on $M$ is locally 
homogeneous, and hence the fundamental group of $M$ is infinite \cite[Corollary 1.1]{BD3}: 
a contradiction. Hence $M$ does not admit any holomorphic projective connection.
\end{proof}

We conjecture that a simply connected compact complex manifold bearing a holomorphic projective connection is 
isomorphic to the complex projective space (endowed with its standard flat projective connection). In particular, 
we conjecture that simply connected compact complex manifolds with trivial canonical bundle do not admit any
holomorphic projective connection. The second part of Proposition \ref{Killing2} which led to the proof of 
Theorem \ref{thm sc} should be seen as a step in this direction. Some other evidence in this direction was provided 
by the main result in \cite{BD2} which says that simply connected compact complex manifolds do not admit 
holomorphic Riemannian metrics; in this case the canonical bundle is automatically trivialized by the volume form 
associated to the holomorphic Riemannian metric.

\section{Appendix}

The aim of this Appendix is to prove the    technical  Lemma \ref{lemma appendix},  used in the proof of Proposition \ref{abelian} (ii),  namely :

{\it $\nabla^{A,B,C,D,E}$ is projectively flat on ${\mathbb T}^3$  if and only if $C\,=\,D$.}

\begin{proof}
For that we shall compute the projective Weyl curvature tensor of $\nabla^{A,B,C,D,E}$.

We start by computing the affine curvature tensor. To simplify the notation in the computation, 
$\nabla^{A,B,C,D,E}$ is denoted simply by $\nabla$.
Recall from \eqref{ct}
that the affine curvature tensor of $\nabla$ is given by the formula $$R(X,Y)Z\,=\,\nabla_X \nabla_Y Z 
-\nabla_Y \nabla_X Z- \nabla_{\lbrack X, Y \rbrack}Z\, .$$
Substituting for $X, Y$ and $Z$ we get the following explicit expressions:
$$R(\frac{\partial}{\partial \tau}, \frac{\partial}{\partial z_1}) \frac{\partial}{\partial z_1}\,=\,
\nabla_{\frac{\partial}{\partial \tau}} \nabla_{\frac{\partial}{\partial z_1}} \frac{\partial}{\partial z_1}-
\nabla_{\frac{\partial}{\partial z_1}} \nabla_{\frac{\partial}{\partial \tau }} \frac{\partial}{\partial z_1}\,=\,
 \nabla_{\frac{\partial}{\partial \tau}} (f_{z_1,z_1}^{z_1} \frac{\partial}{\partial z_1})
$$
$$
- \nabla_{\frac{\partial}{\partial z_1}} (f_{\tau,,z_1}^{\tau} \frac{\partial}{\partial \tau}+
f_{\tau,,z_1}^{z_1}\frac{\partial}{\partial z_1})
$$
$$
=\,
 f_{z_1,z_1}^{z_1}(f_{\tau,,z_1}^{\tau} \frac{\partial}{\partial \tau} + 
f_{\tau,,z_1}^{z_1} \frac{\partial}{\partial z_1}) -f_{\tau, z_1}^{\tau} \nabla_{\frac{\partial}{\partial z_1}}
\frac{\partial}{\partial \tau} -f_{\tau, z_1}^{z_1} \nabla_{\frac{\partial}{\partial z_1}}
\frac{\partial}{\partial z_1}
$$
$$
=\, f_{z_1,z_1}^{z_1}(f_{\tau,,z_1}^{\tau} \frac{\partial}{\partial \tau} + 
f_{\tau,,z_1}^{z_1} \frac{\partial}{\partial z_1})-f_{\tau, z_1}^{\tau} (f_{z_1, \tau }^{\tau}
\frac{\partial}{\partial \tau}+f_{z_1, \tau }^{z_1}
\frac{\partial}{\partial z_1})-f_{\tau, z_1}^{z_1}f_{z_1,z_1}^{z_1} \frac{\partial}{\partial z_1}
$$
$$
=\,\frac{C^2}{4} \frac{\partial}{\partial \tau} -\frac{CE}{4}\frac{\partial}{\partial z_1}\,.
$$
By symmetry we get
$$R(\frac{\partial}{\partial \tau}, \frac{\partial}{\partial z_2}) \frac{\partial}{\partial z_2}\,=\,
\frac{D^2}{4} \frac{\partial}{\partial \tau} -\frac{DE}{4}\frac{\partial}{\partial z_2}\, ,$$
and
$$R(\frac{\partial}{\partial z_1}, \frac{\partial}{\partial z_2}) \frac{\partial}{\partial z_1}\,=\,
\nabla_{\frac{\partial}{\partial z_1}} \nabla_{\frac{\partial}{\partial z_2}} \frac{\partial}{\partial z_1}-
\nabla_{\frac{\partial}{\partial z_2}} \nabla_{\frac{\partial}{\partial z_1 }} \frac{\partial}{\partial z_1}
\,=\, \nabla_{\frac{\partial}{\partial z_1}} (f_{z_1,z_2}^{z_1} \frac{\partial}{\partial z_1} +
f_{z_1,z_2}^{z_2} \frac{\partial}{\partial z_2})
$$
$$
- \nabla_{\frac{\partial}{\partial z_2} } (f_{z_1,z_1}^{z_1} \frac{\partial}{\partial z_1})\,=\,
f_{z_1,z_2}^{z_1}\cdot f_{z_1,z_1}^{z_1} \frac{\partial}{\partial z_1} +f_{z_1,z_2}^{z_2}
\nabla_{\frac{\partial}{\partial z_1}}\frac{\partial} {\partial z_2}-f_{z_1,z_1}^{z_1}
\nabla_{\frac{\partial}{\partial z_2}}\frac{\partial} {\partial z_1}
$$
$$
=\, \frac{C^2}{2} \frac{\partial}{\partial z_1} +( \frac{1}{2}D-C) \nabla_{\frac{\partial}{\partial z_1}}
\frac{\partial} {\partial z_2}\, =\, \frac{C^2}{2} \frac{\partial}{\partial z_1} +(\frac{1}{2}D-C)
(f_{z_1,z_2}^{z_1} \frac{\partial}{\partial z_1}
+ f_{z_1,z_2}^{z_2} \frac{\partial}{\partial z_2})
$$
$$
=\, \frac{C^2}{2} \frac{\partial}{\partial z_1} +(\frac{1}{2}D-C) (\frac{C}{2} \frac{\partial}{\partial z_1} +
\frac{D}{2} \frac{\partial}{\partial z_2})\,=\,
\frac{CD}{4} \frac{\partial}{\partial z_1} + \frac{D^2-2CD}{4} \frac{\partial}{\partial z_2}\, .$$
By symmetry we get $$R(\frac{\partial}{\partial z_1}, \frac{\partial}{\partial z_2})
\frac{\partial}{\partial z_2}\,=\, -R(\frac{\partial}{\partial z_2}, \frac{\partial}{\partial z_1})
\frac{\partial}{\partial z_2}\,=\,- \frac{C^2-2CD}{4} \frac{\partial}{\partial z_1} -
\frac{CD}{4}\frac{\partial}{\partial z_2} .$$

We have
$$R(\frac{\partial}{\partial \tau }, \frac{\partial}{\partial z_1}) \frac{\partial}{\partial z_2}\,=\,
\nabla_{\frac{\partial}{\partial \tau }} \nabla_{\frac{\partial}{\partial z_1}} \frac{\partial}{\partial z_2}-
\nabla_{\frac{\partial}{\partial z_1}} \nabla_{\frac{\partial}{\partial \tau }} \frac{\partial}{\partial z_2}
$$
$$
=\,\nabla_{\frac{\partial}{\partial \tau}} (f_{z_1,z_2}^{z_1} \frac{\partial}{\partial z_1} + f_{z_1,z_2}^{z_2}
\frac{\partial}{\partial z_2}) - \nabla_{\frac{\partial}{\partial z_1}} (f_{\tau ,z_2}^{\tau }
\frac{\partial}{\partial \tau} + f_{\tau ,z_2}^{z_2} \frac{\partial}{\partial z_2})
$$
$$
=\, \frac{1}{2}C (f_{\tau, z_1}^{\tau} \frac{\partial}{\partial \tau} + f_{\tau, z_1}^{z_1}
\frac{\partial}{\partial z_1})+ \frac{1}{2} D( f_{\tau, z_2}^{\tau} \frac{\partial}{\partial \tau} +
f_{\tau, z_2}^{z_2} \frac{\partial}{\partial z_2}) -
f_{z_2, \tau}^{\tau} \nabla_{\frac{\partial}{\partial z_1}}\frac{\partial}{\partial \tau}-
f_{z_2, \tau}^{z_2} \nabla_{\frac{\partial}{\partial z_1}}\frac{\partial}{\partial z_2}
$$
$$
=\, \frac{1}{4} C^2 \frac{\partial}{\partial \tau}+ \frac{1}{4} CE \frac{\partial}{\partial z_1} +
\frac{1}{4} D^2 \frac{\partial}{\partial \tau} + \frac{1}{4} DE \frac{\partial}{\partial z_2} -
\frac{1}{2} D(f_{z_1, \tau}^{\tau} \frac{\partial}{\partial \tau} + f_{z_1, \tau}^{z_1}
\frac{\partial}{\partial z_1})
$$
$$
- \frac{1}{2} E(f_{z_1, z_2}^{z_1} \frac{\partial}{\partial z_1} + f_{z_1, z_2}^{z_2} \frac{\partial}{\partial z_2})
\,=\, \frac{1}{4} (C^2 + D^2- CD) \frac{\partial}{\partial \tau}- \frac{1}{4}DE \frac{\partial}{\partial z_1}.$$
By symmetry, $$R(\frac{\partial}{\partial \tau }, \frac{\partial}{\partial z_2}) \frac{\partial}{\partial z_1}
\,=\, \frac{1}{4} (C^2 + D^2- CD) \frac{\partial}{\partial \tau} - \frac{1}{4}CE \frac{\partial}{\partial z_2}.$$
 
We also compute that
$$R(\frac{\partial}{\partial z_1 }, \frac{\partial}{\partial z_2}) \frac{\partial}{\partial\tau}
\,=\, \nabla_{\frac{\partial}{\partial z_1}} \nabla_{\frac{\partial}{\partial z_2}} \frac{\partial}{\partial \tau }-\nabla_{\frac{\partial}{\partial z_2}} \nabla_{\frac{\partial}{\partial z_1 }} \frac{\partial}{\partial \tau}$$
$$=\, \nabla_{\frac{\partial}{\partial z_1}} (f_{\tau, z_2}^{\tau} \frac{\partial}{\partial \tau} + f_{\tau, z_2}^{z_2} \frac{\partial}{\partial z_2})
- \nabla_{\frac{\partial}{\partial z_2}} (f_{\tau, z_1}^{\tau} \frac{\partial}{\partial \tau} + f_{\tau, z_1}^{z_1} \frac{\partial}{\partial z_1})$$
$$=\,f_{\tau, z_2}^{\tau} (f_{\tau, z_1}^{z_1} \frac{\partial}{\partial z_1} + f_{\tau, z_1}^{\tau} \frac{\partial}{\partial\tau})+ f_{\tau, z_2}^{z_2} (f_{z_1,z_2}^{z_1} \frac{\partial}{\partial z_1} + f_{z_1, z_2}^{z_2} \frac{\partial}{\partial z_2})$$
$$ -f_{\tau, z_1}^{\tau } (f_{\tau ,z_2}^{z_2} \frac{\partial}{\partial z_2} 
+ f_{\tau, z_2}^{\tau } \frac{\partial}{\partial\tau}) - f_{\tau, z_1}^{z_1} (f_{z_1,z_2}^{z_1} \frac{\partial}{\partial z_1} + f_{z_1, z_2}^{z_2} \frac{\partial}{\partial z_2})$$
$$ =\, \frac{DE}{4} \frac{\partial}{\partial z_1} -\frac{CE}{4}\frac{\partial}{\partial z_2}\, ,$$
and 
$$R(\frac{\partial}{\partial \tau }, \frac{\partial}{\partial z_1}) \frac{\partial}{\partial\tau}
\,=\, \nabla_{\frac{\partial}{\partial \tau }} \nabla_{\frac{\partial}{\partial z_1}} \frac{\partial}{\partial \tau }-\nabla_{\frac{\partial}{\partial z_1}} \nabla_{\frac{\partial}{\partial \tau}} \frac{\partial}{\partial \tau}$$
$$=\, \nabla_{\frac{\partial}{\partial \tau}} (f_{\tau, z_1}^{z_1} \frac{\partial}{\partial z_1} +
f_{\tau, z_1}^{\tau} \frac{\partial}{\partial \tau})- \nabla_{\frac{\partial}{\partial z_1}} ( f_{\tau, \tau}^{\tau} \frac{\partial}{\partial \tau}+ f_{\tau, \tau}^{z_1} \frac{\partial}{\partial z_1} + f_{\tau, \tau}^{z_2} \frac{\partial}{\partial z_2})$$
$$=\, f_{\tau, z_1}^{z_1} (f_{\tau, z_1}^{\tau} \frac{\partial}{\partial \tau}+f_{\tau, z_1}^{z_1} \frac{\partial}{\partial z_1})+ f_{\tau, z_1}^{\tau} (f_{\tau, \tau}^{\tau} \frac{\partial}{\partial \tau} + f_{\tau, \tau}^{z_1} \frac{\partial}{\partial z_1} + f_{\tau, \tau}^{z_2} \frac{\partial}{\partial z_2})$$
$$- f_{\tau, \tau}^{\tau} (f_{\tau, z_1}^{z_1} \frac{\partial}{\partial z_1}+f_{\tau, z_1}^{\tau} \frac{\partial}{\partial \tau}) -
f_{\tau, \tau}^{z_1} f_{z_1, z_1}^{z_1} \frac{\partial}{\partial z_1}- f_{\tau, \tau}^{z_2} (f_{z_1, z_2}^{z_1} \frac{\partial}{\partial z_1}+f_{z_1, z_2}^{z_2} \frac{\partial}{\partial z_2})$$
$$=\, \frac{EC}{4} \frac{\partial}{\partial \tau}+ (-\frac{E^2}{4}-\frac{1}{2}C(A+B)) \frac{\partial}{\partial z_1}+ \frac{1}{2}B(C-D) \frac{\partial}{\partial z_2}. $$
By symmetry, $$R(\frac{\partial}{\partial \tau }, \frac{\partial}{\partial z_2}) \frac{\partial}{\partial\tau}
\,=\, \frac{ED}{4} \frac{\partial}{\partial \tau}+ (-\frac{E^2}{4}-\frac{1}{2}D(A+B)) \frac{\partial}{\partial z_2}+ \frac{1}{2}A(D-C) \frac{\partial}{\partial z_1}. $$

Now we compute the Ricci curvature defined in \eqref{ct2}:
$$\rm{Ricci} (\frac{\partial}{\partial z_1},\, \frac{\partial}{\partial z_2})
\,=\, \rm{Ricci} (\frac{\partial}{\partial z_2},\, \frac{\partial}{\partial z_1})\,=\,\frac{1}{4} (C^2+D^2)\, ;$$
$$\rm{Ricci} (\frac{\partial}{\partial \tau}, \,\frac{\partial}{\partial z_1})\,=\,
\rm{Ricci} (\frac{\partial}{\partial z_1},\, \frac{\partial}{\partial \tau})\,=\,\frac{1}{2}CE\, ;$$
$$\rm{Ricci} (\frac{\partial}{\partial z_2},\, \frac{\partial}{\partial \tau})\,=\,
\rm{Ricci} (\frac{\partial}{\partial \tau},\, \frac{\partial}{\partial z_2})\,=\,\frac{1}{2}DE\, ;$$
$$\rm{Ricci} (\frac{\partial}{\partial z_1},\, \frac{\partial}{\partial z_1})\,=\,\frac{1}{4}(C^2+2CD-D^2)\, ;$$
$$\rm{Ricci} (\frac{\partial}{\partial z_2}, \,\frac{\partial}{\partial z_2})\,=\,\frac{1}{4}(D^2+2CD-C^2)\, ;$$
$$\rm{Ricci} (\frac{\partial}{\partial \tau},\, \frac{\partial}{\partial \tau})\,=\,
\frac{1}{2}E^2+ \frac{1}{2}(A+B)(C+D)\, .$$

Recall from \eqref{wt} the formula for the Weyl projective tensor $W$ in dimension three:
$$
W(X,\,Y)Z\,=\, R(X,\,Y)Z-\frac{1}{4} \rm{Tr}R(X,\,Y)Z-\frac{1}{2} \lbrack \rm{Ricci}(Y,\,Z)X-
\rm{Ricci}(X,\,Z)(Y)\rbrack
$$
$$
- \frac{1}{8} \lbrack \rm{Tr}R(Y,\,Z)X-\rm{Tr}R(X,\,Z)(Y)\rbrack\, .$$
Also, recall that the connection $\nabla$ is projectively flat if and only if the tensor $W$ vanishes identically.
The Weyl projective tensor $W$ is anti-symmetric in $(X,\, Y)$ and satisfies the first Bianchi identity in \eqref{bii}.

Since $\rm{Ricci}$ for $\nabla$ is symmetric, it follows that $\rm{Tr}R$ vanishes
identically. Connections with 
symmetric {\rm Ricci} tensor are called equiaffine. The geometrical meaning of it is that there is a parallel 
holomorphic volume form \cite[p.~222, Appendix A.3]{OT}. The above formula for
Weyl projective tensor for $\nabla$ reduces to
$$W(X,\,Y)Z\,=\, R(X,Y)Z-\frac{1}{2} \lbrack \rm{Ricci}(Y,Z)X-\rm{Ricci}(X,Z)(Y) \rbrack.$$
 
The computation for $W(\frac{\partial}{\partial z_1}, \frac{\partial}{\partial z_2}) \frac{\partial}{\partial z_2}$
is as follows:
$$W(\frac{\partial}{\partial z_1}, \frac{\partial}{\partial z_2}) \frac{\partial}{\partial z_2}\,=\, R(\frac{\partial}{\partial z_1}, \frac{\partial}{\partial z_2})\frac{\partial}{\partial z_2}-\frac{1}{2} \lbrack \rm{Ricci}(\frac{\partial}{\partial z_2}, \frac{\partial}{\partial z_2})\frac{\partial}{\partial z_1}-\rm{Ricci}(\frac{\partial}{\partial z_1}, \frac{\partial}{\partial z_2})\frac{\partial}{\partial z_2})\rbrack$$
$$=\,\frac{1}{4}(2CD-C^2) \frac{\partial}{\partial z_1} -\frac{1}{4}CD \frac{\partial}{\partial z_2}-\frac{1}{8}(D^2+2CD-C^2) \frac{\partial}{\partial z_1} + \frac{1}{8} (C^2+D^2)\frac{\partial}{\partial z_2}$$
$$=\,-\frac{1}{8} (C-D)^2 \frac{\partial}{\partial z_1} + \frac{1}{8} (C-D)^2 \frac{\partial}{\partial z_2}\, .$$
Hence
$$W(\frac{\partial}{\partial z_1}, \frac{\partial}{\partial z_2}) \frac{\partial}{\partial z_2}\,=\,
-\frac{1}{8} (C-D)^2 \frac{\partial}{\partial z_1} + \frac{1}{8} (C-D)^2 \frac{\partial}{\partial z_2}\, .$$
Also, $$W(\frac{\partial}{\partial z_1}, \frac{\partial}{\partial z_2}) \frac{\partial}{\partial z_1}\,=\,
R(\frac{\partial}{\partial z_1}, \frac{\partial}{\partial z_2})\frac{\partial}{\partial z_1}-\frac{1}{2} \lbrack \rm{Ricci}(\frac{\partial}{\partial z_2}, \frac{\partial}{\partial z_1})\frac{\partial}{\partial z_1}-\rm{Ricci}(\frac{\partial}{\partial z_1}, \frac{\partial}{\partial z_1})\frac{\partial}{\partial z_2})\rbrack$$
$$=\,\frac{1}{4}CD \frac{\partial}{\partial z_1} + \frac{1}{4}(D^2-2CD)\frac{\partial}{\partial z_2}-\frac{1}{8}(C^2+D^2) \frac{\partial}{\partial z_1} + \frac{1}{8} (C^2+2CD-D^2)\frac{\partial}{\partial z_2}$$
$$=-\frac{1}{8}\, (C-D)^2 \frac{\partial}{\partial z_1} + \frac{1}{8} (C-D)^2 \frac{\partial}{\partial z_2}.$$

In conclusion, $$W(\frac{\partial}{\partial z_1},\, \frac{\partial}{\partial z_2}) \frac{\partial}{\partial z_1}
\,=\, -\frac{1}{8} (C-D)^2 \frac{\partial}{\partial z_1} + \frac{1}{8} (C-D)^2 \frac{\partial}{\partial z_2}\, .$$
We get that $$W(\frac{\partial}{\partial z_1}, \frac{\partial}{\partial z_2}) \frac{\partial}{\partial \tau}\,=\,
R(\frac{\partial}{\partial z_1}, \frac{\partial}{\partial z_2})\frac{\partial}{\partial \tau}-\frac{1}{2}
\lbrack \rm{Ricci}(\frac{\partial}{\partial z_2}, \frac{\partial}{\partial \tau})\frac{\partial}{\partial z_1}
-\rm{Ricci}(\frac{\partial}{\partial z_1}, \frac{\partial}{\partial \tau})\frac{\partial}{\partial z_2})\rbrack
\,=\, 0\, .$$
Hence we have
$$W(\frac{\partial}{\partial z_1},\, \frac{\partial}{\partial z_2}) \frac{\partial}{\partial \tau}\,=\, 0\, . $$

By similar direct computations we get that $$W(\frac{\partial}{\partial \tau},\, \frac{\partial}{\partial z_1})
\frac{\partial}{\partial z_2}\,=\,\frac{1}{8} (C-D)^2 \frac{\partial}{\partial \tau}$$
and
$$W(\frac{\partial}{\partial \tau}, \,\frac{\partial}{\partial z_1}) \frac{\partial}{\partial\tau}\,=\,
\frac{1}{4}(A+B)(D-C)\frac{\partial}{\partial z_1} + \frac{1}{2} B(C-D) \frac{\partial}{\partial z_2}\, .$$
 
Also by direct computation:
$$W(\frac{\partial}{\partial \tau}, \,\frac{\partial}{\partial z_1}) \frac{\partial}{\partial z_1}\,=\,
\frac{1}{8} (C-D)^2 \frac{\partial}{\partial \tau}$$
and
$$W(\frac{\partial}{\partial \tau}, \,\frac{\partial}{\partial z_2}) \frac{\partial}{\partial z_2}
\,=\, \frac{1}{8} (C-D)^2 \frac{\partial}{\partial \tau}\, .$$

Again by a direct computation, $$W(\frac{\partial}{\partial z_2},\, \frac{\partial}{\partial \tau})
\frac{\partial}{\partial\tau}\,=\, \frac{1}{2}A(C-D) \frac{\partial}{\partial z_1}+ 
\frac{1}{4}(A+B)(D-C)\frac{\partial}{\partial z_2}\, .$$

The other components of the Weyl tensor can be obtained using the first Bianchi identity in \eqref{bii}.
Indeed, from $$W(\frac{\partial}{\partial z_1}, \,\frac{\partial}{\partial z_2}) \frac{\partial}{\partial \tau}+ 
W(\frac{\partial}{\partial z_2},\, \frac{\partial}{\partial \tau}) \frac{\partial}{\partial z_1}+ 
W(\frac{\partial}{\partial \tau},\, \frac{\partial}{\partial z_1}) \frac{\partial}{\partial z_2}\,=\,0$$ we infer that 
\begin{equation}\label{ai2}
W(\frac{\partial}{\partial z_2}, \,\frac{\partial}{\partial \tau}) \frac{\partial}{\partial 
z_1}\,=\,-\frac{1}{8}(C-D)^2 \frac{\partial}{\partial \tau}\, .
\end{equation}

Notice that the Weyl projective tensor $W$ does not depend on the parameter $E$. This is due to 
the facts that $W$ is a projective invariant and $\nabla^{A,B,C,D,E}$ is projectively 
equivalent with $\nabla^{A,B,C,D,0}$. Indeed, let $\phi_{\tau}$ be the holomorphic 
one-form on ${\mathbb T}^3$ defined by
$$
\phi_{\tau} (\frac{\partial}{\partial 
\tau})\,=\, \frac{1}{2}E\ \ \text{ and }\ \ \phi_{\tau} (\frac{\partial}{\partial z_i})\, =\, 0
$$
for $i=1,2$. Then 
\begin{equation}\label{ai}
\nabla^{A,B,C,D,E}_XY -\nabla^{A,B,C,D,0}_XY\,=\, \phi_{\tau}(X)(Y)+ \phi_{\tau}(Y)X
\end{equation}
for all holomorphic vector 
fields $X,\, Y$; the identity in \eqref{ai} being tensorial it can be easily
verified for any pair of vectors chosen from the 
basis $(\frac{\partial}{\partial \tau}, \,\frac{\partial}{\partial z_1},\, \frac{\partial}{\partial z_2})$. From
\eqref{ai} it follows immediately that $\nabla^{A,B,C,D,E}$ and $\nabla^{A,B,C,D,0}$ are
projectively equivalent.

{}From \eqref{ai2} and the expression of all components of the Weyl projective tensor,  it follows that $W$ vanishes identically if and only if $C\,=\, D$.

\end{proof} 

\section*{Acknowledgements}

We thank G. Chenevier for a useful discussion about the geometry of Shimura curves.

This work has been supported by the French government through the UCAJEDI Investments in the Future project 
managed by the National Research Agency (ANR) with the reference number ANR2152IDEX201. The first-named author 
is partially supported by a J. C. Bose Fellowship, and school of mathematics, TIFR, is supported by 
12-R$\&$D-TFR-5.01-0500. The second-named author wishes to thank TIFR Mumbai for hospitality.

%%%%%%%%%%%%%%%%%%%%%%%%%%%%%%%%%%%%%%%%%%%%%%%%%%%%%%%%%%%%%%%%%%%%%%%%%%%%%%%%%%%%%%%%%%


\begin{thebibliography}{ZZZZZZ}

\bibitem[Am]{Am} A. M. Amores, Vector fields of a finite type $G$-structure, 
\textit{Jour. Diff. Geom.} \textbf{14} (1980), 1--6.

\bibitem[At]{At} M. F. Atiyah, Complex analytic connections in fibre
bundles, \textit{Trans. Amer. Math. Soc.} \textbf{85} (1957), 181--207.

\bibitem[Be]{Be} A. Beauville, Vari\'et\'es k\"ahleriennes dont la premi\`ere
classe de Chern est nulle, \textit{Jour. Diff. Geom.} \textbf{18} (1983), 755--782.

\bibitem[Bo]{Bo1} F. A. Bogomolov, K\"ahler manifolds with trivial canonical class, \textit{Izv. 
Akad. Nauk. SSSR} {\bf 38} (1974), 11--21; English translation in \textit{ Math. USSR Izv.} {\bf 8} 
(1974), 9--20.

\bibitem[BD1]{BD} I. Biswas and S. Dumitrescu, Branched holomorphic Cartan geometries and 
Calabi-Yau manifolds, \textit{Int. Math. Res. Not.} {\bf 23} (2019), 7428-7458.

\bibitem[BD2]{BD2} I. Biswas and S. Dumitrescu, Holomorphic Riemannian metric and fundamental group, 
\textit{Bull. Soc. Math. Fr.} {\bf 147} (2019), 455--468.

\bibitem[BD3]{BD3} I. Biswas and S. Dumitrescu, Fujiki class $\mathcal C$ and holomorphic geometric structures, 
\textit{Internat. J. Math.} (in press), https://arxiv.org/pdf/1805.11951.pdf.

\bibitem[BD4]{BD4} I. Biswas and S. Dumitrescu, Holomorphic Cartan geometries on complex tori, \textit{C. R. Acad. 
Sci.} {\bf 356} (2018), 316--321.

\bibitem[BM1]{BM} I. Biswas and B. McKay, Holomorphic Cartan geometries and rational 
curves, \textit{Complex Manifolds} \textbf{3} (2016), 145--168.

\bibitem[BM2]{BM1} I. Biswas and B. McKay, Holomorphic Cartan geometries and Calabi-Yau manifolds, \textit{J. 
Geom. and Phys.} \textbf{60} (2010), 661--663.

\bibitem[BM3]{BM2} I. Biswas and B. McKay, Holomorphic Cartan geometries, Calabi-Yau manifolds and rational 
curves, \textit{Diff. Geom. and Appl.} {\bf 28} (2010), 102--106.

\bibitem[Br]{Br} M. Brunella, On holomorphic forms on compact complex threefolds, \textit{Comment. Math. Helv.} 
{\bf 74} (1999), 642--656.

\bibitem[CM]{CM} A. \v{C}ap and K. Melnick, Essential Killing fields of parabolic geometries: projective and 
conformal structures, \textit{Cent. Eur. Journ. Math.} {\bf 11} (2013), 2053--2061.

\bibitem[Cas]{Cas} P. Cascini, Rational curves on complex manifolds, \textit{Milan
Jour. Math.} \textbf{81} (2013), 291--315.

\bibitem[DA]{DA} G. D'Ambra, Isometry groups of Lorentz manifolds, \textit{Invent. Math.} {\bf 92} (1988) 555-565.

\bibitem[DG]{DG} G. D'Ambra and M. Gromov, \textit{Lectures on transformations groups: geometry and dynamics}, 
Surveys in Differential Geometry, Cambridge MA, (1991).

\bibitem[Du1]{Du} S. Dumitrescu, Structures g\'eom\'etriques holomorphes sur les vari\'et\'es complexes 
compactes, \textit{Ann. Sci. \'Ecole Norm. Sup.} {\bf 34} (2001), 557--571.

\bibitem[Du2]{Du1} S. Dumitrescu, Connexions affines et projectives sur les surfaces complexes compactes, 
\textit{Math. Zeit.} {\bf 264} (2010), 301--316.

\bibitem[Du3]{Du2} S. Dumitrescu, Killing fields of holomorphic Cartan geometries, 
\textit{Monatsh. Math.} \textbf{161} (2010), 145--154.

\bibitem[Ei]{Ei} L. P. Eisenhart, {\it Non-Riemannian geometry}, Amer. Math. Soc. Colloquium Publications {\bf 8},
New York (1927).

\bibitem[Fu]{Fu} A. Fujiki, On the structure of compact manifolds in $\mathcal C$, 
\textit{Advances Studies in Pure Mathematics}, \textbf{1}, Algebraic Varieties and 
Analytic Varieties, (1983), 231--302.

\bibitem[Ga]{Ga} J. Gasqui, \'Equivalence projective et \'equivalence conforme, \textit{Annales ENS}, {\bf 12} 
(1979), 101--134.

\bibitem[Gh]{Gh} E. Ghys, D\'eformations des structures complexes sur les espaces homog\`enes de $SL(2, \mathbb{C})$, 
\textit{Jour. Reine Angew. Math.} \textbf{468} (1995), 113--138.


\bibitem[Grom]{Gr} M. Gromov, Rigid Transfomations Groups, Editors D. Bernard and Y. Choquet 
Bruhat, \textit{G\'eom\'etrie Diff\'erentielle}, \textbf{33}, Hermann, (1988), 65--139.

\bibitem[Gu]{Gu} R. C. Gunning, {\textit On uniformization on complex manifolds: the role of connections}, 
Princeton University Press, Princeton, New Jersey, (1978).

\bibitem[HM]{HM} J.-M. Hwang and N. Mok, Uniruled projective manifolds with irreducible reductive 
$G$-structures, \textit{Jour. reine angew. Math.} {\bf 490} (1997), 55--64.

\bibitem[IKO]{IKO} M. Inoue, S. Kobayashi and T. Ochiai, Holomorphic affine 
connections on compact complex surfaces, \textit{J. Fac. Sci. Univ. Tokyo}
\textbf{27} (1980), 247--264.

\bibitem[JR1]{JR1} P. Jahnke and I. Radloff, Projective threefolds with holomorphic normal
projective connections, \textit{Math. Ann.} {\bf 329} (2004), 379--400.

\bibitem[JR2]{JR2} P. Jahnke and I. Radloff, Projective uniformization, extremal Chern classes and quaternionic 
Shimura curves, \textit{Math. Ann.} {\bf 363} (2015), 753--776.

\bibitem[Ka]{Ka} M. Kato, On characteristic forms on complex manifolds, \textit{Jour. Alg.}
{\bf 138} (1991), 424--439.

\bibitem[Kl1]{Kl1} B. Klingler, Un th\'eor\`eme de rigidit\'e non m\'etrique pour les vari\'et\'es 
localement sym\'etriques hermitiennes, \textit{Comment. Math. Helv.} {\bf 76} (2001), 200--217.

\bibitem[Kl2]{Kl2} B. Klingler, Structures affines et projectives sur les surfaces complexes, \textit{Ann. Inst. 
Fourier} {\bf 48} (1998), 441--477.

\bibitem[KO1]{KO} S. Kobayashi and T. Ochiai, Holomorphic projective structures on compact 
complex surfaces, \textit{Math. Ann.} \textbf{249} (1980), 75--94.

\bibitem[KO2]{KO1} S. Kobayashi and T. Ochiai, Holomorphic projective structures on compact 
complex surfaces II, \textit{Math. Ann.} \textbf{255} (1981), 519--521.

\bibitem[KO3]{KO2} S. Kobayashi and T. Ochiai, Holomorphic structures modeled after compact 
hermitian symmetric spaces, in Hano et al. (eds.), \textit{Manifolds and Lie groups}, Papers in 
honor of Yozo Matsushima, Progr. Math. {\bf 14} (1981), 207--221.

\bibitem[KO4]{KO3} S. Kobayashi and T. Ochiai, Holomorphic structures modeled after hyperquadrics, 
\textit{Tohoku Math. J.} {\bf 34} (1982), 587--629.

\bibitem[KV]{KV} D.R. Kohel and H.A. Verrill, Fondamental domains for Shimura curves, \textit{Journ. Th. Nombres 
Bordeaux} {\bf 15} (2003), 205--222.

\bibitem[LB]{LB} H. Lange and C. Birkenhake, \textit{Complex abelian varieties}, Grundlehren der Mathematischen 
Wissenschaften, 302, Springer-Verlag, Berlin, 1992.

\bibitem[LT]{LT} M. L\"ubke and A. Teleman, \textit{The Kobayashi-Hitchin correspondence},
World Scientific Publishing Co., Inc., River Edge, NJ, 1995. 

\bibitem[Ma]{Ma} V. S. Matveev, Proof of the projective Lichnerowicz-Obata conjecture,
{\textit Jour. Diff. Geom.} {\bf 75} (2007), 459-5-02.

\bibitem[Me]{Me} K. Melnick, A Frobenius theorem for Cartan geometries, with applications, \textit{Enseign. Math.} 
{\bf 57} (2011), 57--89.

\bibitem[Moi]{Mo} B. Moishezon, On $n$ dimensional compact varieties with $n$ independent 
meromorphic functions, \textit{Amer. Math. Soc. Transl.} \textbf{63} (1967), 51--77.

\bibitem[MY]{MY} N. Mok and S. K. Yeung, \textit{Geometric realizations of uniformization of conjugates of 
hermitian locally symmetric manifolds} in Complex Analysis and Geometry, ed. V. Ancona and A. Silva, New York, 
Plenum Press (1993), 253--270.

\bibitem[MM]{MM} R. Molzon and K. P. Mortensen, The Schwarzian derivative of maps between manifolds with 
complex projective connections, \textit{Trans. Amer. Math. Soc.} {\bf 348} (1996), 3015--3036.

\bibitem[NO]{NO} T. Nagano and T. Ochiai, On compact Riemannian manifolds admitting essential projective 
transformations, \textit{J. Fac. Sci. Univ. Tokyo} {\bf 33} (1986), 233--246.

\bibitem[No]{No} K. Nomizu, On local and global existence of Killing vector fields, 
\textit{Ann. of Math.} \textbf{72} (1960), 105--120.

\bibitem[OT]{OT} V. Ovsienko and S. Tabachnikov, \textit{Projective
differential geometry old and new}, Cambridge University Press (2005).

\bibitem[Pe]{Pe} V. Pecastaing, On two theorems about local automorphisms of geometric structures, 
\textit{Ann. Inst. Fourier} \textbf{66} (2016), 175--208.

\bibitem[Sha]{Sh} R. Sharpe, {\it Differential Geometry: Cartan's generalization of Klein's 
Erlangen program}, Graduate Texts in Mathematics, 166. Springer-Verlag, New York, 1997.

\bibitem[Shi] {Shi} G. Shimura, On the theory of automorphic functions, \textit{Ann. of Math.} {\bf 70} 
(1959), 101--144.

\bibitem[StG]{StG} H. P. de Saint Gervais, {\it Uniformization of Riemann Surfaces. 
Revisiting a hundred year old theorem}, E.M.S., 2016.

\bibitem[Ue]{Ue} K. Ueno, {\it Classification theory of algebraic varieties and compact complet spaces}, Notes 
written in collaboration with P. Cherenack, Lecture Notes in Mathematics, Vol. 439, Springer-Verlag, Berlin-New 
York, 1975.

\bibitem[Va]{Va} J. Varouchas, K{\"a}hler spaces and proper open morphisms, 
\textit{Math. Ann.} \textbf{283} (1989), 13--52.

\bibitem[Wa]{Wa} H.-C. Wang, Complex Parallelisable manifolds, \textit{Proc. 
Amer. Math. Soc.} \textbf{5} (1954), 771--776.

\bibitem[We]{We} H. Weyl, Infintesimalgeometrie; Einordnung der projektiven und der konformen Auffassung, 
\textit{G\"otingen Nachrichten} (1921) 99-112.

\bibitem[Ya]{Ya} S.-T. Yau, On the Ricci curvature of a compact K\"ahler manifold
and the complex Monge-Amp\`ere equation. I, {\it Comm. Pure Appl. Math.} {\bf 31}
(1978), 339--411.

\bibitem[Ye]{Ye} Y. G. Ye, On Fano manifolds with normal projective connections, \textit{Internat. J. Math.} {\bf 5} 
(1994), 265--271.

\bibitem[Ze]{Ze} A. Zeghib, On discrete projective transformation groups of Riemannian manifolds, {\textit Adv.
Math.} {\bf 297} (2016), 26--53.

\end{thebibliography}
\end{document}